\documentclass[12pt, leqno]{amsart}

\usepackage{shuffle} 

\usepackage{lscape}
\usepackage{graphicx}
\usepackage{amssymb,amsmath,amsthm,amscd}
\usepackage{mathrsfs}
\usepackage{enumerate}
\usepackage[usenames,dvipsnames]{color}
\usepackage[colorlinks=true, pdfstartview=FitV,
 linkcolor=blue,citecolor=blue,urlcolor=blue]{hyperref}

\usepackage[all]{xy}
\usepackage[normalem]{ulem}  

\usepackage{enumitem}

\allowdisplaybreaks[2]
\usepackage{oldgerm}
\usepackage{xspace}

\usepackage{marginnote}

\newcommand{\nc}{\newcommand}
\numberwithin{equation}{section}

\newenvironment{red}{\relax\color{red}}{\relax}
\newenvironment{blue}{\relax\color{Dandelion}}{\hspace*{.5ex}\relax}

\newcommand{\beb}{\begin{blue}}
\newcommand{\eb}{\end{blue}}

\newcommand{\ber}{\begin{red}}
\newcommand{\er}{\end{red}}

\newcommand{\berE}[1]{\begin{red}{}\marginnote{\fbox{\scshape\lowercase{E}}}%
#1}  

\renewcommand{\le}{\leqslant}
\renewcommand{\ge}{\geqslant}

\theoremstyle{plain}
\newtheorem{lemma}{Lemma}[section]
\newtheorem{prop}[lemma]{Proposition}
\newtheorem{theorem}[lemma]{Theorem}

\newcommand{\Prop}{\begin{prop}}
\newcommand{\enprop}{\end{prop}}
\newcommand{\Lemma}{\begin{lemma}}
\newcommand{\enlemma}{\end{lemma}}
\newcommand{\Th}{\begin{theorem}}
\newcommand{\enth}{\end{theorem}}
\newtheorem{corollary}[lemma]{Corollary}
\newcommand{\Cor}{\begin{corollary}}
\newcommand{\encor}{\end{corollary}}
\newtheorem{definition}[lemma]{Definition}
\newtheorem{conjecture}{Conjecture}
\newcommand{\Def}{\begin{definition}}
\newcommand{\edf}{\end{definition}}
\newtheorem{sublemma}[lemma]{Sublemma}
\newcommand{\Sublemma}{\begin{sublemma}}
\newcommand{\ensub}{\end{sublemma}}

\theoremstyle{definition}
\newtheorem{remark}[lemma]{Remark}
\newtheorem{example}[lemma]{Example}
\newtheorem{Convention}[lemma]{Convention}
\newcommand{\Conv}{\begin{Convention}}
\newcommand{\enconv}{\end{Convention}}
\nc{\Conj}{\begin{conjecture}}
\nc{\enconj}{\end{conjecture}}
\nc{\Rem}{\begin{remark}}
\nc{\enrem}{\end{remark}}

\newcommand{\Q}{\mathbb {Q}}

\newcommand{\Z}{{\mathbb Z}}
\newcommand{\A}{{\mathbb A}}

\newcommand{\D}{\mathscr{D}}

\newcommand{\one}{{\bf{1}}}

\newcommand{\seteq}{\mathbin{:=}}

\newcommand{\hd}{{\operatorname{hd}}}

\newcommand{\g}{{\mathfrak{g}}}

\newcommand{\Hom}{\operatorname{Hom}}

\newcommand{\M}{{\mathscr M}}

\newcommand{\hs}{\hspace*}
\newcommand{\ms}{\mspace}

\newcommand{\To}[1][{\hs{2ex}}]{\xrightarrow{\,#1\,}}

\newenvironment{myequation}
{\relax\setlength{\arraycolsep}{1pt}\begin{eqnarray}}
{\end{eqnarray}}
\newenvironment{myequationn}
{\relax\setlength{\arraycolsep}{1pt}\begin{eqnarray*}}
{\end{eqnarray*}}

\nc{\eq}{\begin{myequation}}
\nc{\eneq}{\end{myequation}}
\nc{\eqn}{\begin{myequationn}}
\nc{\eneqn}{\end{myequationn}}

\newcommand{\on}{\operatorname}

\newcommand{\bni}{\be[label=\rm(\roman*)]}
\newcommand{\bnum}{\bni}
\newcommand{\bna}{\be[label=\rm(\alph*)]}

\newcommand{\id}{\on{id}}
\newcommand{\ba}{\begin{array}}
\newcommand{\ea}{\end{array}}

\newcommand{\eqsub}{\begin{subequations}\begin{eqnarray}}
\newcommand{\eneqsub}{\end{eqnarray}\end{subequations}}

\newcommand{\ol}{\overline}

\nc{\la}{\lambda}
\nc{\lam}{\lambda}
\nc{\U}[1][\g]{U_q(#1)}
\nc{\te}{\tilde{e}}
\nc{\tem}{\tilde{e}^{\mathrm{max}}}
\nc{\tei}{\tilde{e}_i}
\nc{\tf}{\tilde{f}}
\nc{\tfm}{\tilde{f}^{\mathrm{max}}}
\nc{\tfi}{\tilde{f}_i}
\nc{\tU}{\widetilde U_q(\g)}
\nc{\tE}{\widetilde{E}}
\nc{\tF}{\widetilde{F}}
\nc{\tK}{\widetilde{K}}
\nc{\tEs}{\widetilde{E}^*}
\nc{\tFs}{\widetilde{F}^*}

\nc{\ttE}{\widetilde{\mathcal{E}}}
\nc{\ttF}{\widetilde{\mathcal{F}}}
\nc{\ttEs}{\ttE^*}
\nc{\ttFs}{\ttF^*}
\nc{\tfs}{\tf^*}
\nc{\tfss}[1]{\tf^{* \hskip 0.05em #1}}
\nc{\tess}[1]{\te^{* \hskip 0.05em #1}}
\nc{\tes}{\te^*}
\nc{\tesm}{\tilde{e}^{* \hskip 0.05em \mathrm{max}}}
\nc{\tfsm}{\tilde{f}^{* \hskip 0.05em \mathrm{max}}}

\nc{\tk}{\tilde{k}}
\nc{\tkone}{\tk_{\ol{1}}}
\nc{\teone}{\tilde{e}_{\ol{1}}}
\nc{\tfone}{\tilde{f}_{\ol{1}}}

\nc{\teibar}{\tilde{e}_{\ol{i}}} \nc{\tfibar}{\tilde{f}_{\ol{i}}}
\nc{\tki}{{\tk}_{\ol {i}}}

\nc{\BZ}{{\mathbb{Z}}}
\nc{\al}{\alpha}
\nc{\tal}{\widetilde{\al}}
\nc{\tch}{\widetilde{h}}
\nc{\qs}{{q}}
\nc{\lan}{\langle}
\nc{\ran}{\rangle}
\nc{\re}{{\mathrm{re}}}
\nc{\wt}{\operatorname{wt}}
\nc{\hwt}{\widehat{\wt}}
\nc{\Ht}{\mathrm{ht}}
\nc{\hHt}{\widehat{\Ht}}
\nc{\ch}{\operatorname{ch}}
\nc{\Um}[1][\g]{U^-_q(#1)}
\nc{\Ue}{U^+_q(\g)}
\nc{\ep}{\varepsilon}
\nc{\hep}{\widehat{\ep}}
\nc{\vphi}{\varphi}
\nc{\sphi}{\varphi^*}
\nc{\eps}{\ep^*}
\nc{\heps}{\hep^{ \hskip 0.2em  *}}

\nc{\nn}{\nonumber}
\def\max{{\mathop{\mathrm{max}}}}
\nc{\vp}{\varpi}
\nc{\cls}{{\operatorname{cl}}}
\nc{\Wt}{{\operatorname{Wt}}}
\nc{\Us}{U'_q(\g)}
\nc{\La}{\Lambda}
\nc{\tLa}{\widetilde\Lambda}
\nc{\ro}{{\rm(}}
\nc{\rf}{{\rm)}}
\nc{\norm}{{\mathrm{norm}}}
\nc{\qbox}{\quad\mbox}
\nc{\braid}{{\mathfrak{B}}}
\nc{\Ad}{\operatorname{Ad}}
\nc{\Aut}{\operatorname{Aut}}
\nc{\dt}[1]{\tilde{\tilde #1}}
\nc{\Sn}{S^{{\mathrm{norm}}}}
\nc{\aff}{{\mathrm{aff}}}
\nc{\rk}{{\mathrm{rk}}}
\nc{\tP}{\widetilde{P}}
\nc{\tW}{\widetilde{W}}
\nc{\Dyn}{\mathrm{Dyn}}
\nc{\tD}{\widetilde{\Delta}}
\nc{\height}[1]{{\operatorname{ht}}(#1)}
\nc{\bl}{\bigl(}
\nc{\br}{\bigr)}
\nc{\Hecke}{\mathrm{H}}
\nc{\HA}{\Hecke^{\mathrm{A}}}
\nc{\HB}{\Hecke^{\mathrm{B}}}
\newcommand{\scbul}{{\,\raise1pt\hbox{$\scriptscriptstyle\bullet$}\,}}
\nc{\vac}{{\phi}}
\nc{\Bt}{\B_\theta(\g)}
\nc{\be}{\begin{enumerate}}
\nc{\ee}{\end{enumerate}}
\nc{\low}{{\mathrm{low}}}
\nc{\upper}{{\mathrm{up}}}
\nc{\Zodd}{\Z_{\mathrm{odd}}}
\nc{\Ft}[1][n]{\mathbb{P}\mathrm{ol}_{#1}}
\nc{\Ftf}[1][n]{\widetilde{\mathbb{P}\mathrm{ol}}_{#1}}
\nc{\KA}{\on{K}^{\mathrm{A}}}
\nc{\KB}{\on{K}^{\mathrm{B}}}
\nc{\Res}{\on{Res}}
\nc{\Fc}[1][{n,m}]{\mathbf{F}_{#1}}
\nc{\tphi}{\tilde{\varphi}}
\nc{\CO}{\mathscr{O}}
\nc{\inte}{\mathrm{int}}
\nc{\Oint}{\mathcal{O}^{\ge0}_{\inte}}
\nc{\vs}{\vspace*}
\nc{\tLt}{\widetilde{L}}
\nc{\tL}{\widetilde{\Lambda}}
\nc{\tu}{\tilde{u}}
\nc{\noi}{\noindent}
\nc{\heigh}{\mathfrak{t}}
\nc{\lowest}{\mathfrak{l}}
\nc{\rootl}{\mathsf{Q}}
\nc{\rlQ}{\rootl}
\nc{\cl}{\colon}

\nc{\uqpg}{U'_q(\mathfrak g)}
\nc{\uq}{\uqpg}
\nc{\Oh}{\widehat{\mathcal{O}}}
\nc{\pn}{p_{\mathfrak{n}}}

\nc{\KLR}{KLR algebra}
\nc{\KLRs}{KLR algebras}
\nc{\cor}{\mathbf{k}}
\nc{\cora}{{\cor(A)}}
\nc{\haut}{\mathrm{ht}}
\nc{\tens}{\mathop\otimes}
\nc{\gmod}{\mbox{-$\mathrm{gmod}$}}
\nc{\gMod}{\mbox{-$\mathrm{gMod}$}}
\nc{\proj}{\mbox{-$\mathrm{proj}$}}
\nc{\gproj}{\mbox{-$\mathrm{gproj}$}}
\nc{\smod}{\mbox{-$\mathrm{mod}$}}
\nc{\Mod}{\mbox{-$\mathrm{Mod}$}}
\nc{\h}{\mathfrak h}
\nc{\Rnorm}{R^{\mathrm{norm}}}
\nc{\Runiv}{R^{\mathrm{univ}}}
\nc{\Rren}{R^{\mathrm{ren}}}

\nc{\Vhat}{\widehat{V}}
\nc{\F}{\mathcal{F}}

\def\T{{\mathcal T}}

\def\Gup{{\mathrm{G^{up}}}}
\def\Glo{{\mathrm{G^{low}}}}
\nc{\fd}[1][A]{\on{\mathrm{flat.dim}_{#1}}}
\nc{\bP}{{\mathbb{P}}}
\nc{\bPh}{\widehat{\mathbb{P}}}
\nc{\bK}[1][{n}]{\widehat{\mathbb{K}}_{#1}}
\nc{\bV}[1][{n}]{\widehat{V}^{\otimes{#1}}}
\nc{\bVK}[1][{n}]{\widehat{V}^{\otimes{#1}}_{\widehat{\mathbb{K}}}}
\nc{\hV}{\widehat{V}}
\nc{\opp}{\mathrm{opp}}
\nc{\col}{\colon}
\nc{\oep}{\epsilon}
\nc{\qtext}{\quad\text}
\nc{\qtextq}[1]{\quad\text{#1}\quad}
\nc{\longtwoheadrightarrow}[1][]{\xymatrix{\ar@{->>}[r]^-{{#1}}&}}
\nc{\epiTo}[1][]{\longtwoheadrightarrow[{#1}]}
\nc{\epito}{\twoheadrightarrow}
\nc{\monoTo}[1][]{\xymatrix{\ar@{>->}[r]^-{{#1}}&}}
\nc{\sym}{\mathfrak{S}}
\nc{\inp}[1]{{({#1})_{\mathrm{n}}}}
\nc{\rtl}{\rootl}
\nc{\wtd}{\widetilde}
\nc{\etens}{\boxtimes}
\nc{\ds}[1]{\mathrm{d}(#1)}
\nc{\rmat}[1]{{\mathbf{r}}_%
{\mspace{-2mu}\raisebox{-.6ex}{${\scriptstyle{#1}}$}}}
\nc{\rmats}[1]{{\mathbf{r}}_%
{\mspace{-2mu}\raisebox{-.6ex}{${\scriptscriptstyle{#1}}$}}}
\nc{\shc}{\mathcal{C}}
\nc{\shs}{\mathcal{S}}
\nc{\Fct}{{\on{Fct}}}
\nc{\tC}{\widetilde{\shc}}
\nc{\Zp}{\Z_{\ge0}}
\nc{\tPhi}{\widetilde{\Phi}}
\nc{\tT}{{\widetilde{\T}}}
\nc{\Ob}{\on{Ob}}
\nc{\bwr}{\mbox{\large$\wr$}}
\nc{\Img}{\on{Im}}
\nc{\Ab}{\mathcal{A}^{\mathrm{big}}}
\nc{\Sb}{\mathcal{S}^{\mathrm{big}}}
\nc{\As}{\mathcal{A}}
\nc{\Ss}{\mathcal{S}}
\nc{\ntens}{\widetilde{\otimes}}
\nc{\hR}{\widehat{R}}
\nc{\nconv}{\mathop{\mbox{\large $\odot$}}}
\nc{\snconv}{\mbox{\scriptsize$\odot$}}
\nc{\ts}{\tilde{s}}
\nc{\sho}{\mathcal{O}}
\nc{\bc}{\begin{cases}}
\nc{\ec}{\end{cases}}
\nc{\slnh}{{\widehat{\mathfrak{sl}}_N}}
\nc{\UA}{U_q'(\slnh)}
\nc{\KR}{R_K}
\nc{\cQ}{\mathcal{Q}}
\nc{\Irr}{\mathcal{I}rr}
\nc{\tQ}{\widetilde{\cQ}}
\nc{\bs}{\mathbf{s}}
\nc{\bL}{\mathbb{L}}
\nc{\tg}{\tilde{g}}

\nc{\conv}{\mathbin{\mbox{\large $\circ$}}}
\nc{\shconv}{\mathbin{\large\diamond}}
\nc{\sconv}{\mathbin{\large\Delta}}
\nc{\stens}{\mathbin{\large\Delta}}
\nc{\hconv}{\mathbin{\nabla}}
\nc{\htens}{\mathbin{\nabla}}

\nc{\Rm}{R^{\mathrm{ren}}}

\nc{\bQ}{\ol{Q}}

\nc{\de}{\on{\textfrak{d}\ms{1mu}}}

\nc{\xmono}{\ar@{>->}}
\nc{\xepi}{\ar@{->>}}
\nc{\db}[1]{\raisebox{-.5ex}[2ex][1.8ex]{$#1$}}
\nc{\wb}[1]{\mbox{$\rule[-1.1ex]{0ex}{2ex}#1$}}
\nc{\univ}{\mathrm{univ}}
\nc{\rM}{{}^*\mspace{-2mu}M}
\nc{\lM}{M^*}
\nc{\uqm}{\uq\smod}
\nc{\tR}{\widetilde{R}_{\gamma,\beta}}
\nc{\tx}{\tilde{x}}
\nc{\bi}{\mathbf{i}}
\nc{\ttau}{\widetilde{\tau}}

\nc{\tEnd}{\on{\widetilde{E}nd}}
\nc{\tHom}{\on{\widetilde{H}om}}

\nc{\K}{{J}}
\nc{\Kex}{{\K}_{\mathrm{ex}}}
\nc{\Kfr}{{\K}_{\mathrm{f\mspace{.01mu}r}}}
\nc{\coro}{\cor}
\nc{\tB}{\widetilde{B}}
\nc{\seed}{\mathscr{S}}

\nc{\up}{\mathrm{up}}
\nc{\bfa}{\mathbf{a}}
\nc{\bfb}{\mathbf{b}}
\nc{\bfc}{\mathbf{c}}
\nc{\bfm}{\mathbf{m}}
\nc{\hbfm}{ \widehat{\mathbf{m}}}

\newlength{\mylength}
\nc{\ov}[1]{\overline{#1}}
\nc{\Wlmj}[3]{\W_{#2,#3}^{(#1)}}
\nc{\Mkl}[2]{\M_\ttww(#1,#2)}
\nc{\mqs}{(-q^2)}
\nc{\Cquiver}{\upsigma}

\nc{\mut}[1]{{\mu}_{\mspace{-2mu}\raisebox{-.5ex}{${\scriptstyle{#1}}$}}}

\nc{\Kt}{\mathcal K_t}
\nc{\KT}{\mathbb{K}_t}
\nc{\yim}{y_{i,m}}
\nc{\yjm}{y_{j,m}}
\nc{\yjp}{y_{j,p}}
\nc{\yimp}{y_{i,m+1}}
\nc{\yjmp}{y_{j,m+1}}

\nc{\Refl}{\mathscr{S}}
\nc{\Reflinv}{{\Refl}^{-1}}

\nc{\Refn}{\mathsf{S}}

\nc{\catC}{\mathscr C}
\nc{\catA}{\mathcal A}
\nc{\shift}{{\mathrm T}}
\nc{\rE}{ \mathsf{E} }
\nc{\rW}{ \mathcal{W} }
\nc{\rES}{ \mathcal{E} }

\nc{\brd}{\sigma} 
\nc{\into}{\xymatrix@C=3ex{{}\ar@{^{(}->}[r]&{}}}
\nc{\dual}{\D\ms{1.5mu}}
\nc{\cdual}{D}
\nc{\cat}[1][{\g}]{\catC_{#1}^0}

\nc{\catCO}{{\catC_\g^0}}
\nc{\catCOD}{{\catC_\g^{0, \ddD}}}
\nc{\catCQ}{{\catC_{\qQ}}}
\nc{\catCQd}{{\catC_{\widetilde{\qQ}}}}
\nc{\catCD}{{\catC_{\ddD}}}
\nc{\catCDK}{{\catC_{\ddD, \iK}}}
\nc{\Li}{{\La^\infty}}
\nc{\sig}{{\sigma(\g)}}
\nc{\sigZ}{{\sigma_0(\g)}}
\nc{\sigQ}{{\sigma_\qQ(\g)}}
\nc{\sigQd}{{\sigma_{\widetilde{\qQ}}(\g)}}
\nc{\phiQd}{\phi_{\widetilde{\qQ}}}
\nc{\sigD}{{\sigma_\ddD(\g)}}

\nc{\ZZ}{{\mathbf{Z}}}
\nc{\sP}{{\mathsf{P}}}
\nc{\sV}{{\mathsf{V}}}
\nc{\rxw}{{\underline{w}}}
\nc{\rxwz}{{\underline{w_0}}}
\nc{\boten}[1]{\overrightarrow{\bigotimes_{#1}}}
\nc{\cmA}{{\mathsf{A}}}
\nc{\cmC}{{\mathsf{C}}}
\nc{\ddD}{{\mathcal{D}}}
\nc{\ddDQ}{{\ddD_Q}}
\nc{\ddDQd}{{\ddD_{\widetilde{Q}}}}
\nc{\qQ}{{\mathcal{Q}}}
\nc{\gf}{{\g_{\mathrm{fin}}}}
\nc{\Df}{{\Delta_{\mathrm{fin}}}}
\nc{\If}{{I_{\mathrm{fin}}}}
\nc{\cmAf}{{\cmA_{\mathrm{fin}}}}
\nc{\weyl}{{\mathsf{W}}}
\nc{\weylf}{{\mathsf{W}_{\mathrm{fin}}}}
\nc{\sg}{{\mathfrak{S}}}
\nc{\weylA}{{\mathsf{W}_\cmA}}
\nc{\weylC}{{\mathsf{W}_\cmC}}
\nc{\Deg}{\mathrm{Deg}}
\nc{\Di}{\Deg^\infty}
\nc{\KRc}{{K_{q=1}(R_\cmC\gmod)}}
\nc{\prD}{{\Delta^+}}
\nc{\prDf}{{\Delta^+_{\mathrm{fin}}}}
\nc{\nrD}{{\Delta^-}}
\nc{\prDA}{{\Delta^+_\cmA}}
\nc{\prDC}{{\Delta^+_\cmC}}
\nc{\nrDC}{{\Delta^-_\cmC}}
\nc{\n}{{\mathfrak{n}}}
\nc{\Rt}{\mathsf{L}} 
\nc{\Cp}{\mathsf{V}} 
\nc{\cuspS}{{\mathsf{S}}}
\nc{\st}[1]{\{{#1}\}}
\nc{\bst}[1]{\bigl\{{#1}\bigr\}}
\nc{\WS}{Quantum affine Schur-Weyl duality\xspace}
\nc{\CWS}{Quantum affine Weyl-Schur duality}
\nc{\zz}{{{\mathbf{z}}}}
\nc{\wlP}{\mathsf{P}}
\nc{\twlP}{\widetilde{\wlP}}
\nc{\wl}{\wlP}
\nc{\clp}{{\mathrm{cl}}}
\nc{\wlPc}{{\wlP_\clp}}
\nc{\awlP}{\widehat{\mathsf{P}}}
\nc{\dM}{\mathsf{M}}
\nc{\dC}{\mathsf{C}}
\nc{\cC}{\mathcal{C}}
\nc{\lR}{\widetilde{{R}}}
\nc{\zero}{\mathrm{zero}}
\nc{\PD}{principal }

\nc{\prtl}[1][J]{\rootl_{#1}^+}
\nc{\hL}{\widehat{\Rt}}
\nc{\hF}{\widehat{\F}}
\nc{\Proof}{\begin{proof}}
\nc{\QED}{\end{proof}}
\nc{\e}{\mathrm{e}}

\nc{\Aff}{\mathrm{Aff}}
\nc{\rT}{\mathcal{T}}		
\nc{\rr}{rationally renormalizable\xspace}
\nc{\RA}{{R_\cmA}}		
\nc{\RC}{{R_\cmC}}		
\nc{\proolim}[1][]{\mathop{``{\varprojlim}{\mbox{''}}}\limits_{#1}}
\nc{\qtq}[1][\text{and}]{\quad\text{#1}\quad}

\newcommand{\iB}{{{_i}B(\infty)}} 			 
\newcommand{\Bi}{{B_i (\infty)}} 			 
\newcommand{\cT}{\mathrm{T}} 		
\newcommand{\cTs}{\mathrm{T}^{\ms{1mu}*}}
\newcommand{\tcT}{\widetilde{\mathrm{T}}} 		 
\newcommand{\tcTs}{\widetilde{\mathrm{T}}^{\ms{1mu}*}}
\newcommand{\lT}{\mathcal{T}} 		
\newcommand{\lTs}{\mathcal{T}^{\ms{1mu}*}}
\nc{\corh}{\widehat{\cor}}
\nc{\ang}[1]{\langle{#1}\rangle}
\nc{\rc}{renormalizing coefficient\xspace}
\nc{\cz}{{\cor[z^{\pm1}]}}
\nc{\tp}{\ms{1.5mu}{\widetilde{p}}\ms{2mu}}
\nc{\G}{\mathcal{G}}
\nc{\cc}{\mathfrak{c}}
\nc{\rsP}{{\Phi_\g}}
\nc{\rsX}{{X_\g}}
\nc{\rs}{ \mathsf{s} }
\nc{\Dynkin}{\mathsf{D}}
\nc{\Dat}{\sigma}
\nc{\hf}{\xi}

\nc{\hBi}{\widehat{B}(\infty)}
\nc{\hBvi}{{\widehat{B}_\vt(\infty)}}
\nc{\hBsi}{{\widehat{B}_\sigma(\infty)}}
\nc{\iK}{\mathsf{K}}
\nc{\cBg}[1][\g]{\widehat{B}_{#1}(\infty)}
\nc{\cBsg}[1][\g]{\widehat{B}_{#1}(\infty)^*}
\nc{\cBgk}[1][\g]{\widehat{B}_{#1}^{\iK}(\infty)}
\nc{\cb}{\mathbf{b}}
\nc{\cI}{ \widehat{I} }
\nc{\cIf}{{\widehat{I}_{\mathrm{fin}}}}
\nc{\cIz}{{\widehat{I}_{0}}}
\nc{\cJ}{ \widehat{J} }
\nc{\sB}{\mathcal{B}}
\nc{\sBk}{\sB_{\iK}}
\nc{\sBdk}{\sB_{\ddD, \iK}}
\nc{\sBD}{\sB_{\ddD}}
\nc{\sBkg}{\sBk(\g)}
\nc{\cs}{\star}

\nc{\cd}{\mathrm{D}}
\nc{\cm}{\mathbf{m}}
\nc{\MS}{\mathsf{MS}}
\nc{\hMS}{\widehat{\mathsf{MS}}}
\nc{\rS}{\mathbf{S}}
\nc{\rSs}{\rS^*}
\nc{\brS}{\overline{\rS}}

\nc{\crI}{ \mathsf{I}}
\nc{\crhI}{\mathscr{S} }
\nc{\crD}{\mathsf{D}}
\nc{\crc}{\mathsf{c}}
\nc{\crB}{\mathscr{P}_n}
\nc{\cru}{\mathsf{u}}
\nc{\crsh}{\mathsf{sh}}
\nc{\bfs}{\mathbf{s}}
\nc{\crBB}[1]{\mathscr{P}_#1}

\nc{\gc}{{\g_{\cmC}}}
\nc{\cL}{\mathcal{L}}
\nc{\cLD}{{\cL_\ddD}}
\nc{\ccL}{\mathscr{L}}
\nc{\ccLD}{{\ccL_\ddD}}
\nc{\clen}{\mathsf{len}}

\nc{\hv}{\mathsf{1}}
\nc{\qt}[1]{\quad\text{#1}}
\nc{\snoi}{\smallskip\noindent}
\nc{\mnoi}{\medskip\noindent}

\nc{\ul}[1]{\underline{#1}}
\nc{\dul}[1]{\underline{\underline{#1}}}
\nc{\qh}{\qedhere}
\nc{\ca}{\mathsf{v}}
\nc{\sck}[1][k]{\ms{8mu}{\raisebox{-1.3ex}{$\scriptstyle{#1}$}\hs{-1.4ex}{\succ}}\ms{4mu}}
\nc{\scke}[1][k]{\ms{8mu}{\raisebox{-1.3ex}{$\scriptstyle{#1}$}\hs{-1.4ex}%
{\succcurlyeq}}\ms{4mu}}

\nc{\edot}{\emptyset}
\nc{\Nf}{N_{\gf}}
\nc{\fin}{\mathrm{fin}}
\nc{\trg}{\scalebox{.7}{$\triangle$}}
\nc{\ake}[1][1ex]{\rule[-#1]{0ex}{1ex}}
\nc{\akew}[1][1ex]{\rule[-1ex]{#1}{0ex}}
\nc{\akeu}[1][1ex]{\rule[#1]{0ex}{1ex}}

\nc{\bg}{\mathscr{B}}
\newcommand{\B}{B(\infty)}
\nc{\ipi}{{{_i}\pi}}
\nc{\pii}{{\pi_i}}
\nc{\Ld}{\mathcal{P}}
\nc{\Dc}{\Upsilon}
\nc{\DcL}{\mathcal{L}}
\nc{\bR}{\mathsf{R}}
\nc{\bRs}{\bR^*}
\nc{\rR}{\mathrm{R}}
\nc{\bfi}{\mathbf{i}}
\nc{\bfj}{\mathbf{j}}
\nc{\bfk}{\mathbf{k}}
\nc{\vt}{\vartheta}
\nc{\nP}{\Upsilon}
\nc{\hP}{\widehat{\nP}}
\nc{\fF}{ \widetilde{\mathrm{F}}}
\nc{\fE}{ \widetilde{\mathrm{E}}}
\nc{\fR}{\mathscr{R}}
\nc{\fT}{\mathscr{T}}
\nc{\orb}{\mathrm{orb}}
\nc{\fdr}{\mathrm{r}}

\title[Braid group action on extended crystals]{Braid group action on extended crystals}

\author[E. Park]{Euiyong Park}
\thanks{The research of E. Park was supported by the National Research Foundation of Korea(NRF) Grant funded by the Korea Government(MSIT)(NRF-2020R1F1A1A01065992 and NRF-2020R1A5A1016126).}
\address[E. Park]{Department of Mathematics, University of Seoul, Seoul 02504, Korea}
\email
{epark@uos.ac.kr}

\keywords{Braid groups, Extended crystals, Hernandez-Leclerc categories, Quantum affine algebras} 
\subjclass[2010]{05E10, 05E18, 17B37} %

\date{July 21, 2022}

\begin{document}

 \maketitle

\begin{abstract}
In the paper, we prove that there exists a braid group action on the extended crystal $\hBi$ of finite type. 
The extended crystal $\hBi$ and its braid group action are investigated from the viewpoint of crystal similarity.
We then interpret the braid group action on $\hBi$ in the Hernandez-Leclerc category $\catCO$.

\end{abstract}

\tableofcontents

\section*{Introduction}

The notion of \emph{extended crystals} is introduced in \cite{KP22} for studying the module category of a quantum affine algebra from the viewpoint of the crystal basis theory. The crystals are one of the most powerful tools for studying quantum groups in a combinatorial way, and they appear naturally in a large number of applications in various research areas (see \cite{Kas90, K91, K93, K94}, and see also \cite{HK02, Lu93} and the references therein). We denote by $U_q(\g)$ the quantum group associated with a generalized Cartan matrix $\cmA = (a_{i,j})_{i,j\in I}$ and  by $B(\infty)$ the crystal of the negative half $U_q^-(\g)$.
The extended crystal is defined as
\begin{align*}
	\hBi \seteq  \left\{  (b_k)_{k\in \Z } \in \prod_{k\in \Z} B(\infty) \biggm| b_k =\hv \text{ for all but finitely many $k$} \right\},
\end{align*}
where $\hv$ is the highest weight vector of $B(\infty)$. 
The \emph{extended crystal operators} $\tF_{i,k}$ and $\tE_{i,k}$ ($(i,k)\in \cI \seteq I \times \Z$) are defined in terms of the usual crystal operators $\tf_i$, $ \tes_i$, $ \tfs_i$ and $\te_i$ for the crystal $B(\infty)$, and the weight $\hwt$ for $ \hBi$ is defined by using usual weight function $\wt$ on $B(\infty)$ (see Section \ref{Sec: extended crystal} for details).
As a usual crystal, one can define an $\cI$-colored graph structure on the extended crystal $\hBi$ by using the operators $\tF_{i,k}$ ($(i,k)\in \cI$). 
It is shown in \cite{KP22} that $\hBi$ is connected as an $\cI$-colored graph and  that there exist interesting symmetries $ \cd$ and $\star$ on $\cBg$ compatible with extended crystal operators.

Let $\catC_\g$ be the category of finite-dimensional integrable modules over a quantum affine algebra $U_q'(\g)$, where $q$ is an indeterminate. 
The category $\catC_\g$ has been studied widely in various research fields including representation theory, geometry and mathematical physics  (see \cite{AK97, CP94, FR99, HL10, Kas02, Nak01} for example).
The category $\catC_\g$ has a distinguished subcategory $\catCO$, which is called a \emph{Hernandez-Leclerc category} (\cite{HL10}). 
We choose a \emph{complete duality datum} $\ddD = \{ \Rt_i \}_{i\in \If} $ for $\catCO$ (see Section \ref{subsec: HL}).
Let $\sB(\g)$ be the set of the isomorphism classes of simple modules in $\catCO$.  
It is proved in \cite{KP22} that $\sB(\g)$ has the categorical crystal structure defined by
$$
\ttF_{i,k} (M) \seteq    (\dual^k \Rt_i) \htens M ,\qquad \ttE_{i,k} (M) \seteq    M \htens (\dual^{k+1} \Rt_i)
$$
for $ M \in \sB(\g)$ and $ (i, k)\in \cIf$
and that $\sB(\g)$ is isomorphic to the extended crystal $\cBg[\gf]$ of the crystal $B_\gf(\infty)$, i.e., 
$$
\Phi_\ddD \col  \cBg[\gf]  \buildrel \sim \over \longrightarrow \sBD(\g).
$$
 Here $\dual$ is the right dual functor in $\catCO$ and $\gf$ is the simple Lie algebra of simply-laced finite type given in \eqref{Table: root system}.
Under the isomorphism $\Phi_\ddD$ between $\sB(\g)$ and $\cBg[\gf]$, the categorical crystal operators $\ttF_{i,k} $ and $\ttE_{i,k} $ correspond to the extended crystal operators $\tF_{i,k}$ and $\tE_{i,k}$, and the dual functor $\dual$ matches with the operator $\cd$. 
This extended crystal isomorphism allows us to study simple modules in $\catCO$ in terms of  the extended crystal $\cBg[\gf]$, which is a new combinatorial approach to the category $\catCO$ from the viewpoint of crystals.

\smallskip

In this paper, we prove that there exists a braid group action on the extended crystal $\hBi$ associated with a \emph{finite Cartan matrix} and study several properties of the braid group action.
Let $\cmA = (a_{i,j})_{i,j\in I}$ be a Cartan matrix of finite type, and let $\hBi$ be the extended crystal of the crystal $B(\infty)$ associated with $\cmA$.
We denote by $\bg_\cmA$ the \emph{generalized braid group} (or \emph{Artin-Tits group}) defined by the generators $r_i$ ($i\in I$) and the following defining relations:
$$
\underbrace{r_{i}r_{j}r_{i}r_{j} \cdots}_{m(i,j) \text{ factors}} = \underbrace{r_{j}r_{i}r_{j}r_{i} \cdots}_{m(j,i) \text{ factors}} 
\qquad \text{ for $i,j\in I$ with $i\ne j$,}
$$ 
where $m(i,j)$ is the integer determined by $\cmA$ (see \eqref{Eq: m(i,j)} for the precise definition).
We simply call $\bg_\cmA$ the braid group associated with $\cmA$. 

 For each $i\in I$, we define bijections 
$$
 \bR_i: \hBi  \buildrel \sim \over \longrightarrow \hBi \quad \text{ and }\quad    \bRs_i: \hBi  \buildrel \sim \over \longrightarrow \hBi
 $$
using usual crystal operators and the \emph{Saito crystal reflections} (\cite{Saito94}) on $B(\infty)$ (see  Section \ref{Sec: braid group actions}). We prove that $\bR_i$ and $\bRs_i$ are inverse to each other and that they satisfy the braid group relations for $\bg_\cmA$ (Theorem \ref{Thm: main1}). Thus we have the action of $\bg_\cmA$ on $\hBi$ via the bijections $\bR_i$. 
In the course of proofs, \emph{PBW bases} are used crucially. 
The actions $\bR_i$ and $\bRs_i$ commute with the operator $\cd$, i.e.,  $\bR_i \circ \cd = \cd \circ   \bR_i $ and $\bRs_i \circ \cd = \cd \circ   \bRs_i $, and 
they are compatible with the reflection $s_i$ on the weight lattice via the weight function $\hwt$, i.e.,
$$
\hwt ( \bR_i (\cb) ) = s_i ( \hwt(\cb) ), \quad  \hwt ( \bRs_i (\cb) ) = s_i ( \hwt(\cb) ) \qquad \text{for any $\cb \in \hBi$}
$$
(see Lemma \ref{Lem: R R*}).
Thus the braid group action $\bR_i$ can be understood as a natural extension of the Saito crystal reflection on $B(\infty)$ to the extended crystal $\hBi$. 
Let $\rR(w_0)$ be the set of all reduced expression of the longest element $w_0$ (see \eqref{Eq: reduced ex}). 
We show that, for any    $\bfi \in \rR(w_0)$, 
$$
\bR_\bfi = \cd \circ \zeta \qquad \text{ and } \qquad \bRs_\bfi = \cd^{-1} \circ \zeta,
$$ 
where $\zeta$ is the involution defined in \eqref{Eq: i->i*}.
In particular, unless the Cartan matrix $\cmA$ is of type $A_n$ ($n\in \Z_{>1}$), $D_{n}$ ($n$ is odd) or $E_6$, we have
$$
\bR_\bfi = \cd   \qquad \text{ and } \qquad \bRs_\bfi = \cd^{-1}.
$$ 
Hence the central elements of the braid group $\bg_\cmA$ act as $\cd^t$ on $\hBi$ for some $t\in \Z$. 
As a $\bg_\cmA$-set, $\hBi$ is not transitive. It is conjectured that $ \hBi$ is faithful as a $\bg_\cmA$-set for any finite Cartan matrix $\cmA$ (Remark \ref{Rmk: tran}).
We also conjecture that a braid group action with similar properties exists for a generalized Cartan matrix of \emph{arbitrary type.}   

We next investigate the extended crystal $\hBi$ from the viewpoint of \emph{similarity of crystals} (\cite{K96}). Applying the Dynkin diagram folding (\cite[Section 5]{K96}) for $B(\infty)$ to extended crystals, we obtain an analogue of \cite[Theorem 5.1]{K96} for extended crystals.   
Let $\sigma$ be a Dynkin diagram automorphism given in \eqref{Eq: Folding1}. We denote by  $\hBsi$
the extended crystal associated with the Cartan matrix folded by $\sigma$ and by $ \hBi^\sigma$ the set of fixed points of $\hBi$ under the action $\sigma$. 
Proposition \ref{Prop: folding} says that there exists an crystal isomorphism  
$$
\hP_\sigma :  \hBsi \buildrel \sim \over \longrightarrow \hBi^\sigma,
$$
where the extended crystal operator for $\hBi^\sigma$ is defined as a product of usual extended crystal operators of $\hBi$ in the same $\sigma$-orbit. Moreover, the isomorphism $\hP_\sigma$ is compatible with the braid group actions, i.e., 
$$
\hP_\sigma \left(r_j'(\cb)\right) = \fdr^\sigma_j \left(\hP_\sigma (\cb) \right) \qquad \text{ for any }j,
$$
where $ r_j'$ is a generator of the braid group $\bg_\sigma$ associated with the Cartan matrix folded by $\sigma$ and 
$ \fdr^\sigma_j $ is defined as a product of usual generators of $\bg_\cmA$ in the same $\sigma$-orbit (see \eqref{Eq: folding braid generators}).

We finally interpret the braid group action on $\hBi$ in the Hernandez-Leclerc category $\catCO$.
It is announced in \cite{KKOP21A} that 
there is an action of the braid group $\bg_\cmAf$ on the quantum Grothendieck ring of $\catCO$ and that 
there are monoidal autofunctors on the localization  $\T_N$ which give the braid group actions at the Grothendieck ring.  
We remark that  $\T_N$ can be regarded as a graded version of $\catCO$ for affine type $A^{(1)}_{N-1}$.   
It is conjectured in \cite[Section 5]{KKOP21A} that such functors $\fR_i$ ($i\in \If$) exist for any arbitrary quantum affine algebra (see Conjecture \ref{Conj: R_i}).
Under the assumption that the conjecture holds, Proposition \ref{Prop: braid for CgO} tells us that the braid group action $\bR_i$ on $\hBi$ is a crystal-theoretic shadow of the conjectural functor $\fR_i$, i.e., 
$$
\Phi_\ddD ( \bR_i (\cb) ) = \fR_i ( \Phi_\ddD(\cb) )  \qquad \text{ for any $\cb \in \cBg[\gf]$ and $i\in \If$.}
$$

\medskip

This paper is organized as follows. 
In Section \ref{Sec:Preliminaries}, we review briefly the notion of extended crystals and Hernandez-Leclerc categories. 
In Section \ref{Sec: crystals and PBW}, we recall the Saito crystal reflections and a connection to PBW bases.  
In Section \ref{Sec: braid group actions}, we define a braid group action on the extended crystal $\hBi$ and investigate its properties. 
We then study the similarity for $\hBi$ in Section \ref{Sec: similarity}, and interpret the braid group action for $\hBi$ in the Hernandez-Leclerc category $\catCO$ in Section \ref{Sec: HL cat}.

\vskip 2em

\section{Preliminaries} \label{Sec:Preliminaries}

\subsection{Crystals} \

In this subsection, we briefly review the notion of crystals (see \cite{Kas90, K91, K93, K94}, and see also \cite{HK02}).
Let $I $ be a finite index set. 

\begin{definition}
A quintuple $ (\cmA,\wlP,\Pi,\wlP^\vee,\Pi^\vee) $ is called a  (symmetrizable) {\it Cartan datum} if it
consists of
\bna
	\item a generalized  \emph{Cartan matrix} $\cmA=(a_{ij})_{i,j\in I}$, 
	\item a free abelian group $\wlP$, called the {\em weight lattice},
	\item $\Pi = \{ \alpha_i \mid i\in I \} \subset \wlP$,
	called the set of {\em simple roots},
	\item $\wlP^{\vee}=
	\Hom_{\Z}( \wlP, \Z )$, called the \emph{coweight lattice},
	\item $\Pi^{\vee} =\{ h_i \in \wlP^\vee \mid i\in I\}$, called the set of {\em simple coroots}, 
\ee
which satisfy the following properties: \bnum
\item $\lan h_i, \alpha_j \ran = a_{ij}$ for $i,j \in I$,
\item $\Pi$ is linearly independent over $\Q$,
\item for each $i\in I$, there exists $\Lambda_i \in \wlP$, called a \emph{fundamental weight}, such that $\lan h_j,\Lambda_i \ran =\delta_{j,i}$ for all $j \in I$.
\item there is a symmetric bilinear 
form $( \cdot \, , \cdot )$ on $\wlP$ satisfying 
$(\al_i,\al_i)\in\Q_{>0}$ and 
$ \lan h_i,  \lambda\ran = {2 (\alpha_i,\lambda)}/{(\alpha_i,\alpha_i)}$. 
\ee
\end{definition}

We denote by $ \rlQ \seteq \bigoplus_{i \in I} \Z \alpha_i$ and by $\prD$ the set of  positive roots.
The \emph{Weyl group} $\weyl$ associated with $\cmA$ is the subgroup of $\mathrm{Aut}(\wlP)$ generated by  
$$ 
s_i(\lambda) \seteq  \lambda - \langle h_i, \lambda \rangle \alpha_i  
$$ 
for any $i\in I$. 
In this paper, the standard notation for Dynkin diagrams given in \cite{Kac} is used except type $E_k$ ($k=6,7,8$). 
For the case of type $E_k$ ($k=6,7,8$), we follow the notation for Dynkin diagrams appeared in  \cite[Appendix A.1]{KKOP22}.
For $w \in \weyl$, $\ell(w)$ denotes the length of $w$, and $\rR(w)$ is the set of all reduced expressions of $w$, i.e., 
\begin{align} \label{Eq: reduced ex}
\rR(w) := \{ (i_1, i_2, \ldots, i_t) \in I^t \mid w = s_{i_1} s_{i_2} \cdots s_{i_t} \}
\end{align}
where $t = \ell(w)$.
We assume that $\cmA$ is of finite type. Let $w_0$ be the longest element of $\weyl$.
One can show that, for any $ (i_1, i_2,  \ldots, i_\ell) \in \rR(w_0) $, we have 
\begin{align} \label{Eq: rxw rot}
	(i_{2}, i_{3}, \ldots, i_\ell, i_{1}^* )\in \rR(w_0) \quad \text{ and } \quad    (i_\ell^*,  i_{1}, i_{2}, \ldots, i_{\ell-1} )  \in \rR(w_0),
\end{align}	 
where $i^*$ is defined by 
\begin{align} \label{Eq: *}
\al_{i^*} = - w_0(\al_i)
\end{align}

Let $U_q(\g)$ be the \emph{quantum group} associated with $(\cmA, \wlP,\wlP^\vee, \Pi, \Pi^{\vee})$, and $U_q^-(\g)$ be the subalgebra of $U_q(\g)$ generated by $f_i$ ($i\in I$) (see \cite[Chapter 3]{HK02} for details).
Let $\A=\Z[q,q^{-1}]$, and  $\A_0$ be the subring of $\Q(q)$ consisting of rational functions which are regular at $q=0$.
For each $i \in I$, we denote by $\tf_i$ and $\te_i$ the  Kashiwara operators  on $U_q^-(\g)$ (\cite[(3.5.1)]{K91}), and 
set
\begin{align*}
	&L(\infty) : = \sum_{l \in \Z_{\ge0}, \ i_1,\ldots, i_l \in I} \A_0 \tf_{i_1} \cdots \tf_{i_l} \hv \subset U_q^-(\g),  \ \  \overline{L(\infty)} : = \{ \overline{x} \in U_q^-(\g) \mid x \in L(\infty) \}, \\
	&B(\infty) : = \{ \tf_{i_1} \cdots \tf_{i_l} \hv \mod qL(\infty) \mid l \in \Z_{\ge0}, \ i_1,\ldots, i_l \in I \} 
	\subset L(\infty) /qL(\infty),
\end{align*}
where
$^- : U_q(\g) \buildrel \sim \over \rightarrow  U_q(\g)$ is the $\Q$-algebra automorphism defined by $\overline{e_i}=e_i, \ \overline{f_i}=f_i, \ \overline{q^h}=q^{-h},\ \text{and} \  \overline{q}=q^{-1}$.
We call $B(\infty)$ the crystal of $U_q^-(\g)$. When we need to emphasize $B(\infty)$ as a $U_q(\g)$-crystal, we write $B_\g(\infty)$ instead of $B(\infty)$.

Let $G^\low$ be the inverse of the $\Q$-linear isomorphism  $(\Q\otimes_\Z U_\A^-(\g)) \cap L(\infty) \cap \overline{L(\infty)} \buildrel \sim\over \longrightarrow L(\infty)/qL(\infty)$.
The set 
$$\Glo (\infty) : =\{\Glo(b) \in U_\A^-(\g) \mid b \in B(\infty) \}$$
is an $\A$-basis of $U_\A^-(\g)$, which is called the \emph{lower global basis} (or \emph{canonical basis}). Then
we have the dual basis with respect to the \emph{Kashiwara bilinear form} (\cite[Proposition 3.4.4]{K91})
$$\mathbf \Gup (\infty) : =\{\Gup(b)  \mid b \in B(\infty) \},$$
which is called the  \emph{upper global basis} (or \emph{dual canonical basis}).
The $\Q(q)$-antiautomorphism $*$ of $U_q(\g)$ given by
$	(e_i)^* = e_i$, $(f_i)^* = f_i$, and  $(q^h)^*  = q^{-h}$
 provides another crystal operators $ \te_i^*$ and $ \tf_i^*$.
For any $b \in B(\infty)$, we set 
$$
\tem_i (b) := \te_i^{\ep_i(b)} (b) \quad \text{ and }\quad  \tesm_i (b) := \te_i^{* {\hskip 0.05em}  \eps_i(b)} (b).
$$
Here we define $\ep_i(b) := \max \{ k \ge 0 \mid \te_i^k (b)\ne 0 \}$ and $ \vphi_i(b) := \ep_i(b) + \langle h_i, \wt(b) \rangle$ (resp.\ $\ep_i^*(b) := \max \{ k \ge 0 \mid \te_i^{*k} (b)\ne 0 \}$ and   $ \vphi_i^*(b) := \eps_i(b) + \langle h_i, \wt(b) \rangle$).
For more details of crystals, we refer the reader to \cite{Kas90, K91, K93, K94} and \cite[Chapter 4]{HK02}.

We define 
\begin{align} \label{Eq: i->i*}
	\zeta:  I \buildrel \sim\over \longrightarrow  I,  \qquad i \mapsto i^*,
\end{align}
where $i^*$ is given in \eqref{Eq: *}.
If $\g$ is of type $A_n$ ($n \in \Z_{>1}$), $D_{n}$ ($n$ is odd), and $E_6$, then $\zeta$  is given as follows:
\bna
\item (Type $A_n$)   $i^* =  n+1-i $, 
\item (Type $D_n$) 
$i^* =
\begin{cases}
	n-(1-\xi)& \text{ if  $n$ is odd and $i=n-\xi$ ($\xi=0,1$)},  \\
	i& \text{ otherwise,}
\end{cases}
$
\item (Type $E_6$)  The map $i \mapsto i^*$ is determined by 
$$i^* =
\begin{cases}
	6 & \text{ if } i=1,  \\
	i & \text{ if } i=2,4,  \\
	5 & \text{ if } i=3.  
\end{cases}
$$
\ee
Otherwise, $\zeta$ is the identity.
Note that we use the index set of the Dynkin diagram of type $E_6$ given in  \cite[Appendix A.1]{KKOP22}.
The involution $\zeta$ induces an involution on the crystal $ B(\infty)$, which is also denoted by $\zeta$.

\begin{example} \label{Ex: multisegment}

We briefly review the \emph{multisegment realization} of $B(\infty)$ of type $A_n$ (see \cite{V01} and see also \cite{SST18} and \cite[Section 7.1]{KP22}). This realization will be used as an example in the paper.
Let $[a,b]$ be an interval for $  1 \le a \le b \le n $.  A \emph{multisegment} is a multiset of $[a,b]$'s, and we denote by $\MS_n$ the set of multisegments. 
If  $a > b$, then we define $[a,b]\seteq\emptyset$. We write $[a] = [a,b]$ if $a=b$. 
For any multisegment $\cm = \{ m_1, \ldots, m_k \}$, we write $\cm = m_1 + m_2 + \cdots + m_k$.
The crystal structure of $\MS_n$ is described in \cite[Section 7.1]{KP22}, which is isomorphic to $B(\infty)$. 
We follow the description of \cite[Section 7.1]{KP22}.

\end{example}

\subsection{Extended crystals} \label{Sec: extended crystal} \

In this subsection, we briefly review the notion of extended crystals introduced in \cite{KP22}.

We keep the notations given in the previous subsection.
We set 
\begin{align*}
	\hBi \seteq  \Bigl\{  (b_k)_{k\in \Z } \in \prod_{k\in \Z} B(\infty) \bigm| b_k =\hv \text{ for all but finitely many $k$} \Bigr\}.
\end{align*}
For any $\bfb = ( b_k)_{k\in \Z} \in \hBi$, we sometimes write $ \bfb = ( \ldots, b_2, b_1, \underline{b_0}, b_{-1}, \ldots  )$, where the underlined element is located at the 0-position.
We define $\one \seteq  (\hv)_{k\in \Z} \in \hBi$, which can be viewed as a highest weight vector of $\hBi$. Set  $\cI \seteq  I \times \Z$, and  
let $(i,k)\in \cI$ and  $\cb = (b_k)_{k\in \Z} \in \hBi $. Define   
$$	\hwt(\cb)  \seteq  \sum_{k \in \Z}  \wt_k(\cb), $$
where $\wt_k(\cb) \seteq  (-1)^k \wt(b_k)$, and set 
$$
\hBi_\beta := \{ \cb \in \hBi \mid \hwt(\cb) = \beta \} \qquad \text{ for any $\beta \in \rlQ$.}
$$
We define
\begin{align*}
	\hep_{i,k}(\cb) \seteq   \ep_{i, k}(\cb) - \eps_{i,k+1}(\cb), 
\end{align*}
where 
$ \ep_{j,t}(\cb) \seteq   \ep_j(b_t)$ and    $ \eps_{j,t}(\cb) \seteq   \eps_j(b_t)$ for $(j,t) \in \cI$.

The \emph{extended crystal operators} 
\begin{align*}
	\tF_{i,k} \col  \hBi \longrightarrow \hBi \qtq	\tE_{i,k} \col  \hBi \longrightarrow \hBi
\end{align*}
are defined by
\begin{equation} \label{Eq: tE and tF}
	\begin{aligned}
		\tF_{i,k}(\cb) & \seteq  
		\begin{cases}
			(\cdots , b_{k+2},  \  b_{k+1} , \ \tf_i( b_k), \ b_{k-1}, \cdots ) & \text{ if } \hep_{ i,k} (\cb) \ge 0,\\
			(\cdots , b_{k+2},  \ \tes_i (b_{k+1} ), \ b_k, \ b_{k-1}, \cdots ) & \text{ if }  \hep_{ i,k} (\cb) < 0 ,
		\end{cases}
		\\
		\tE_{i,k}(\cb) & \seteq  
		\begin{cases}
			(\cdots , b_{k+2},  \  b_{k+1} , \ \te_i( b_k), \ b_{k-1}, \cdots ) & \text{ if } \hep_{i,k} (\cb) >  0,\\
			(\cdots , b_{k+2},  \ \tfs_i (b_{k+1} ), \ b_k, \ b_{k-1}, \cdots ) & \text{ if }  \hep_{i,k} (\cb) \le  0,
		\end{cases}
	\end{aligned}
\end{equation}
for any $(i, k) \in  \cI$ and $\cb =   (b_k)_{k\in \Z} \in \hBi$. 
Note that $\tF_{ i,k }$ is the inverse of $\tE_{ i,k }$ (\cite[Lemma 4.2]{KP22}).
It is shown in \cite[Section 4]{KP22} that the extended crystals $\hBi$ enjoy similar properties to usual crystals.

For $p \in \Z$ and  $\cb = (b_k)_{k\in \Z} \in \hBi$, we define $\cd^p (\cb) = (b_k')_{k\in \Z} \in \hBi$ by
\eq
b_k'  = b_{k-p} \qquad \text{ for any $k\in \Z$},
\label{def:dual}
\eneq
which gives a bijection 
\begin{align*}
	\cd^p \col  \hBi \longrightarrow  \hBi.
\end{align*}
We remark that 
$
\cd^p ( \tF_{i,k} (\cb) ) = \tF_{i,k+p} (   \cd^p  (\cb) )
$
for any $p\in \Z$ and $ (i,k) \in \cI $.

The extended crystal $\hBi $ has the $\cI $-colored graph structure defined by $\tF_{i,k}$ for $(i,k) \in \cI$, i.e.,
$$
\cb \To[{(i,k)}]\cb' \quad  \text{ if and only if} \quad  \cb' = \tF_{i,k}( \cb).
$$

\begin{prop} [{\cite[Lemma 4.4]{KP22}}] \label{Lem: connectedness}
As an $\cI$-colored graph, $\hBi $ is connected.
\end{prop}

\begin{example} \label{Ex: multi_EC}
We keep the notation given in Example \ref{Ex: multisegment}.
We define 
$$
\hMS_n  := \Bigl\{  (\bfm_k)_{k\in \Z } \in \prod_{k\in \Z} \MS_n \bigm|  \bfm_k = \emptyset \text{ for all but finitely many $k$} \Bigr\}. 
$$
Since $\MS_n $ is isomorphic to $ B(\infty)$, 
$\hMS_n$ has the extended crystal structure induced from the crystal structure of $\MS_n$.	
	
\end{example}

\subsection{Categorical crystals of Hernandez-Leclerc categories} \label{subsec: HL}\ 

In this subsection, we briefly review the categorical crystals of Hernandez-Leclerc categories developed in \cite{KP22}.
This subsection will be used only in Section \ref{Sec: HL cat}.

We denote by $\catC_\g$ the category of finite-dimensional integrable modules over $U_q'(\g)$ associated with an affine Cartan matrix $\cmA$. 
We set $ I_0 \seteq I \setminus \{ 0 \}$. Here we refer the reader to \cite[Section 2.3]{KKOP20A} for the choice of $0$. 
We write $M^{\otimes k} \seteq \underbrace{M \tens \cdots \tens M}_{\ake[.5ex] k\text{-times}}$
for $k\in\Z_{\ge0}$. 
For any module $X$ of finite length, $\hd(X)$ is the head of $X$ and, for the sake of simplicity, $M \htens N$ denotes the head of $M\tens N$ for $M,N\in \catC_\g$. 
A simple module $N$ is \emph{real} if $N \tens N $ is simple.
For any $M\in \catC_\g$,   $\dual (M)$ denotes the right dual of $M$, which is extended to $\dual^k (M)$ for all $k \in \Z$.
We denote by $\catCO$ the \emph{Hernandez-Leclerc category}, which is a full subcategory of $\catC_\g$ with certain conditions (see \cite{HL10} and see also \cite[Section 2.2]{KP22} for details).

We call  $\ddD \seteq \{ \Rt_i \}_{i\in J} $
 a \emph{strong duality datum} associated with a simply-laced finite Cartan matrix $\cmC = (c_{i,j})_{i,j\in J}$
 if it satisfies the following: 
\bna
\item $\Rt_i$ is a  root module for any $i\in J$, 
\item $\de(\Rt_i, \dual^k(\Rt_j)) = - \delta(k= 0)\, c_{i,j}$
for any $k\in \Z$ and $i,j\in J$ with $i\not=j$,
\ee
where a simple module  $N$ is a root module if it is real and $ \de\bl N, \dual^k N\br  =  \delta(k=\pm 1)$ for any $k\in\Z$. 
For a strong duality datum $\ddD = \{ \Rt_i \}_{i\in J} $, 
one can consider the corresponding \emph{quantum affine Schur-Weyl duality functor} 
$$
\F_\ddD \col \RC\gmod \longrightarrow \catCO.
$$
Here $\RC$ is the \emph{symmetric quiver Hecke algebra} corresponding to $\cmC$ (see \cite{KKK18A, KKOP20A} for details). We denote by $\catCD$  the smallest full subcategory of $\catCO$ satisfying 
\bna
\item   $\catCD$ contains $\F_\ddD( N )$ for any simple  $R_{\cmC}$-module $N$, 
\item  $\catCD$ is stable by taking subquotients, extensions, and tensor products. 
\ee
A strong duality datum $\ddD$ is said to be \emph{complete} if for any simple module $M$ in $\catCO$, there are simple modules $M_k $ in $ \catCD$ $(k\in \Z)$ such that 
\bna
\item $M_k$ is isomorphic to the trivial module $ \one$ for all but finitely many $k$,

\item $M $ is isomorphic to $ \hd ( \cdots \tens \dual^2 M_2 \tens \dual M_1 \tens \ M_0 \tens \dual^{-1} M_{-1} \tens \cdots  ).$
	\ee

Let $\gc$ be the Lie algebra associated with $\cmC$.
If $\ddD$ is complete, then the simple Lie algebra $\gc$ has to be of the type $X_\g$ given in the table \eqref{Table: root system} (see \cite[Proposition 6.2]{KKOP20A}).
In this case, we write $\gf$, $\If$, $\cmAf$ etc.\ instead of $\gc$, $J$, $\cmC$ etc. 
\renewcommand{\arraystretch}{1.5}
\begin{align} \label{Table: root system} \small
	\begin{array}{|c||c|c|c|c|c|c|c|} 
		\hline
		\text{Type of $\g$} & A_n^{(1)}  & B_n^{(1)} & C_n^{(1)} & D_n^{(1)} & A_{2n}^{(2)} & A_{2n-1}^{(2)} & D_{n+1}^{(2)}  \\
		&(n\ge1)&(n\ge2)&(n\ge3)&(n\ge4)&(n\ge1)&(n\ge2)&(n\ge3)\\
		\hline
		\text{Type $X_\g$} & A_n & A_{2n-1}    & D_{n+1}   &  D_n & A_{2n} & A_{2n-1} & D_{n+1}  \\
		\hline
		\hline
		\text{Type of $\g$} & E_6^{(1)}  & E_7^{(1)} & E_8^{(1)} & F_4^{(1)} & G_{2}^{(1)} & E_{6}^{(2)} & D_{4}^{(3)}  \\
		\hline
		\text{Type $X_\g$} & E_6 & E_{7}    & E_{8}   & E_6 & D_{4} & E_{6} & D_{4}  \\
		\hline
	\end{array}
\end{align} 

From now on, we assume that $\ddD = \{ \Rt_i \}_{i\in \If} $ is a complete duality datum. 
Let $\sB(\g)$ be the set of the isomorphism classes of simple modules in $\catCO$.  
It is proved in \cite{KP22} that $\sB(\g)$ has the categorical crystal structure defined by
$$
\ttF_{i,k} (M) \seteq    (\dual^k \Rt_i) \htens M,  \qquad  \ttE_{i,k} (M) \seteq    M \htens (\dual^{k+1} \Rt_i)
$$
for $ M \in \sB(\g)$ and $ (i, k)\in \cIf$,  
and that $\sB(\g)$ is isomorphic to the extended crystal $\cBg[\gf]$. 
We write $\sBD(\g)$ for $\sB(\g)$ when considering $\sB(\g)$ with $\ttF_{i,k}$ and $\ttE_{i,k}$.   
Let us briefly explain the isomorphism between $\cBg[\gf]$ and $\sBD(\g)$.

Let $ B_\ddD$ be the set of the isomorphism classes of simple modules in $\catCD$. 
It is shown in \cite{LV09} that the set of the isomorphism classes of simple modules in $ R_\gf\gmod$ forms a crystal isomorphic to the crystal $B_{\gf}(\infty)$. 
As $\F_\ddD$ preserves simple modules, $\F_\ddD$ induces a bijective map 
\begin{align} \label{Eq: Fd LD}
\cL_\ddD\col  B_{\gf}(\infty) \buildrel \sim \over \longrightarrow B_\ddD 
\end{align}
(see \cite[Lemma 3.2]{KP22}).
For any $\cb = (b_k)_{k\in \Z}  \in \cBg[\gf]$,  define 
\begin{align*}
	\ccLD(\cb) \seteq  \hd \left( \bigotimes_{k=+\infty }^{-\infty} \dual^k L_k \right) =  \hd ( \cdots \tens \dual^2 L_2 \tens \dual L_1 \tens  L_0 \tens \dual^{-1} L_{-1} \tens \cdots ), 
\end{align*}
where $ L_k =    \cLD( b_{k} ) $ for $k\in \Z$. Then the map $\Phi_\ddD \col  \cBg[\gf]  \buildrel \sim \over \longrightarrow \sBD(\g)$  defined by 
$$ 
\Phi_\ddD ( \cb) \seteq  \ccLD(\cb) \qquad \text{ for any $\cb \in  \cBg[\gf] $}
$$
is an extended crystal isomorphism, i.e., 
\begin{align*}
	\Phi_\ddD(  \tF_{i,k}  ( \cb) ) =  \ttF_{i,k} (  \Phi_\ddD(    \cb )), \qquad \Phi_\ddD(  \tE_{i,k}  ( \cb) ) =  \ttE_{i,k} (  \Phi_\ddD(    \cb ))
\end{align*}
for $(i,k)\in \cIf $ and $\cb \in  \cBg[\gf] $ (see \cite[Theorem 5.9]{KP22}).

\begin{example} \label{Ex: ECI}
Let $U_q'(\g)$ be the quantum affine algebra of affine type $A_2^{(1)}$. In this case, $\gf$ is of type $A_2$ and $I_0 = \{1,2\}$. 
We review briefly the extended crystal isomorphism between $\hMS_2$ and $ \sBD(\g)$ given in \cite[Section 7.4]{KP22}.
Let  
\begin{align*}
	\crBB{2} \seteq (\Z_{\ge0})^{\oplus  \crhI_2},
\end{align*}
where $\crhI_2 \seteq  \{ (i, a) \in I_0 \times \Z \mid   a-i \equiv 1 \bmod 2  \} $. 
 We regard $ \crhI_2$ as a subset of $\crBB{2}$.
For $(i,a) \in \crBB{2}$ and $k\in \Z$, we set 
$$
\cdual^k(i,a) := 
\begin{cases}
	(i, a + 3k) & \text{ if $k$ is even},\\
	(3-i, a + 3k) & \text{ if $k$ is odd}.
\end{cases}
$$
We define $\gamma_k : \MS_2 \rightarrow  \crBB{2}$ by 
$$
 \gamma_k ( a[2] + b[12] + c[1]) = a \cdual^k( 1,2 ) + b \cdual^k( 2,1 ) +  c \cdual^k( 1,0 ).
$$
Then we have the bijection
\begin{align} \label{eq: gamma}
\gamma: \hMS_2 \buildrel \sim \over \longrightarrow  \crBB{2}
\end{align}
defined by $ \gamma( \hbfm) := \sum_k \gamma_k(\bfm_k)$ for any $ \hbfm = (\bfm_k)_{k\in \Z} \in \hMS_2$.

For any $\la = \sum_k (i_k, a_k) \in  \crBB{2}$, we denote by $V(\la)$ the simple module in $\catC_\g$ whose \emph{affine highest wight} is $ \sum_{k} (i_k, (-q)^{a_k})$ (see \cite[Theorem 2.2]{KP22} for affine height weights). Note that 
$V(i,a) = V(\varpi_{i})_{(-q)^a}$ and $ \dual^k \bl V( i,a)\br\simeq  V\bl \cdual^k ( i,a) \br $ for $k\in \Z$.

We now choose a complete duality datum $\ddD = \{  \Rt_1, \Rt_2   \} $, where $\Rt_1 := V(\varpi_1) $ and $\Rt_2 := V(\varpi_1)_{(-q)^2} $, and define
$$
\Phi_\ddD: \hMS_2 \longrightarrow \sBD(\g)
$$ 
by $ \Phi_\ddD ( \hbfm)  := V(  \gamma (\hbfm)) $ for any $ \hbfm \in \hMS_2$. Then $\Phi_\ddD$ becomes an extended crystal isomorphism (see \cite[Section 7.4]{KP22}). 

We remark that $ \crBB{2}$ has also an extended crystal structure and its crystal operators are described in \cite[Section 7.3]{KP22} in a combinatorial manner.

\end{example}

\vskip 2em

\section{Crystals and PBW bases} \label{Sec: crystals and PBW}

In this section, we shall investigate several properties about crystals and PBW bases.

From now on, we assume that $\cmA = (a_{i,j})_{i,j\in I}$ is a Cartan matrix of finite type. Let  $B(\infty)$ be the crystal of the negative half of the quantum group associated with $\cmA$. 
For $i\in I$, we denote Lusztig's braid symmetries by 
$$
\lT_i := T_{i,-1}' \quad \text{ and } \quad \lTs_i := T_{i,1}'',
$$ 
where $T_{i,-1}'$ and $ T_{i,1}''$ are the symmetries defined in \cite[Chapter 37.1]{Lu93}. 
Let $\ell := \ell(w_0)$ and let $\bfi = (i_1, i_2, \ldots, i_\ell) \in \rR(w_0)$.
For any $\epsilon = \pm 1$, $c \in \Z_{ \ge 0}$ and $k=1, \ldots, \ell$, we set 
$\beta_k := s_{i_1} \cdots s_{i_{k-1}}(\al_{i_k}) \in \prD$ and 
$$
f_{\bfi, \epsilon} (\beta_k) ^{(c)} := 
\begin{cases}
	\lT_{i_1}\lT_{i_2} \cdots \lT_{i_{k-1}}(f_{i_k}^{(c)})  	 & \text{ if } \epsilon = -1, \\
	\lTs_{i_1}\lTs_{i_2} \cdots \lTs_{i_{k-1}}(f_{i_k}^{(c)})   & \text{ if } \epsilon = 1.
\end{cases}
$$ 
For any $\bfa=(a_1, \ldots, a_\ell) \in \Z_{\ge 0}^{\oplus \ell}$, we define 
$$
f_{\bfi, \epsilon} (\bfa) := 
\begin{cases}
	f_{\bfi, \epsilon} (\beta_\ell) ^{(a_\ell)} f_{\bfi, \epsilon} (\beta_{\ell-1}) ^{(a_{\ell-1})} \cdots 	f_{\bfi, \epsilon} (\beta_{1}) ^{(a_{1})}  & \text{ if } \epsilon = -1, \\
	f_{\bfi, \epsilon} (\beta_{1}) ^{(a_{1})}  \cdots f_{\bfi, \epsilon} (\beta_{\ell-1}) ^{(a_{\ell-1})}   f_{\bfi, \epsilon} (\beta_\ell) ^{(a_\ell)}   & \text{ if } \epsilon = 1.
\end{cases}
$$ 
The set $\{ f_{\bfi, \epsilon} (\bfa) \}_{\bfa \in \Z_{\ge0}^{ \oplus \ell}}$ is called the \emph{PBW basis} of $ U_q^-(\g)$ with respect to $\bfi$ and $\epsilon$. 

We now consider the dual version of the PBW basis. For any $\bfa \in \Z_{\ge 0}^{\oplus \ell}$,  we define
$$
f^{\up}_{\bfi, \epsilon} (\bfa) := f_{\bfi, \epsilon} (\bfa)/( f_{\bfi, \epsilon} (\bfa),f_{\bfi, \epsilon} (\bfa) ),
$$
where $(-,-)$ is the Kashiwara bilinear form defined in \cite[Proposition 3.4.4]{K91}. 
 Then the set $\{ f^{\up}_{\bfi, \epsilon} (\bfa)  \}_{\bfa \in \Z_{\ge0}^{\oplus\ell}}$ becomes a basis of $ U_q^-(\g)$, which is called the \emph{dual PBW basis} with respect to $\bfi$ and $\epsilon$. Note that  
$$
f^{\up}_{\bfi, \epsilon} (\bfa) = 
\begin{cases}
	f_{\bfi, \epsilon}^{\up} (\beta_\ell) ^{\{a_\ell\}} f_{\bfi, \epsilon}^{\up} (\beta_{\ell-1}) ^{\{a_{\ell-1}\}} \cdots 	f_{\bfi, \epsilon}^{\up} (\beta_{1}) ^{\{a_{1}\}}  & \text{ if } \epsilon = -1, \\
	f_{\bfi, \epsilon}^{\up} (\beta_{1}) ^{\{a_{1}\}}  \cdots f_{\bfi, \epsilon}^{\up} (\beta_{\ell-1}) ^{\{a_{\ell-1}\}}   f_{\bfi, \epsilon}^{\up} (\beta_\ell) ^{\{a_\ell\}}   & \text{ if } \epsilon = 1,
\end{cases}
$$
where we define $ f^\up_{\bfi, \epsilon}(\beta_k)^{\{a\}} := q_{i_k}^{a(a-1)/2} f^\up_{\bfi, \epsilon}(\beta_k)^a$. Note that the dual PBW vectors $f^\up_{\bfi, \epsilon}(\beta_k)^{\{a\}}$ are contained in $\Gup(\infty)$.

\medskip

For any $i\in I$, we set 
\begin{align*}
	\iB &:= \{b\in \B \mid \ep_i(b) = 0  \}, \\
    \Bi &:= \{b\in \B \mid \eps_i(b) = 0  \}. 
\end{align*}
The \emph{Saito crystal reflections} (see \cite{Saito94}) on the crystal $B(\infty)$ are defined as follows 
\begin{align*}
	\cT_i: &\ \iB \rightarrow \Bi, \qquad \cT_i(b) := \tf_i ^{ \vphi_i^*(b)} \te_i^{* \hskip 0.1em  \ep_i^*(b)}(b), \\
	\cTs_i: &\ \Bi \rightarrow \iB, \qquad \cTs_i(b) := \tf_i ^{* \hskip 0.1em  \vphi_i(b)} \te_i^{\ep_i(b)}(b).	
\end{align*} 
Note that $ \cT_i \circ \cTs_i = \id $ and $ \cTs_i \circ \cT_i = \id $. The crystal reflections $ \cT_i$ and $\cTs_i$ are the crystal counterparts of the braid symmetries $ \lT_i $ and $\lTs_i$. 
Let $\ipi: \B \rightarrow \iB$ and $\pii: \B \rightarrow \Bi$ be the surjective maps defined by 
$$
\ipi(b) := \tem_i (b) \quad \text{ and } \quad  \pii(b) := \tesm_i (b) \qquad \text{ for } b\in \B ,
$$
and we set 
\begin{align*}
\tcT_i	&:= \cT_i \circ \ipi : \B \longrightarrow \Bi \subset \B, \\
\tcTs_i	&:= \cTs_i \circ \pii : \B \longrightarrow \iB \subset \B.
\end{align*}	
For any $\bfj = (j_1, \ldots, j_t) \in \rR(w)$, we define 
$$
\tcT_\bfj := \tcT_{j_1} \tcT_{j_2} \cdots \tcT_{j_t} \quad \text{ and }\quad   \tcTs_\bfj := \tcTs_{j_1} \tcTs_{j_2} \cdots \tcTs_{j_t}. 
$$

\begin{definition} [Lusztig datum] \label{Def: LD}
For $b \in B(\infty)$ and $ \bfi = (i_1, \ldots, i_\ell) \in \rR(w_0) $, we define 
\begin{align*}
\DcL_{\bfi, \epsilon}(b) :=  
\begin{cases}
	\left( \eps_{i_\ell} \left( \tcTs_{i_{\ell-1}} \cdots \tcTs_{i_2}\tcTs_{i_1}(b) \right),\,  \ldots, \ \eps_{i_2} \left( \tcTs_{i_1}(b) \right),\   \eps_{i_1}(b)   \right) & \text{ if } \epsilon = -1, \\
	\left( \ep_{i_1}(b),\ \ep_{i_2} \left( \tcT_{i_1}(b) \right), \  \ldots,\  \ep_{i_\ell} \left( \tcT_{i_{\ell-1}} \cdots \tcT_{i_2}\tcT_{i_1}(b)  \right)   \right)  & \text{ if } \epsilon = 1.
\end{cases}
\end{align*}
\end{definition}

We remark that
\begin{align} \label{Eq: Gup PBW}
\Gup(b) \equiv f^{\up}_{\bfi, \epsilon} (\bfa) \mod qL(\infty) \qquad \text{for $b \in B(\infty)$,} 
\end{align}
where $\bfa = \DcL_{\bfi, \epsilon}(b) $ (see \cite{Lu90,Lu93}, and see also \cite[Theorem 4.29]{Kimura12} and \cite[Section 3.2.2]{Kimura17}).
By the PBW theory, it is easy to see that the map
\begin{align*}
	\DcL_{\bfi, \epsilon} :   B(\infty)  \buildrel \sim \over \longrightarrow \Z_{\ge 0}^{\oplus\ell}, \qquad b \mapsto \DcL_{\bfi, \epsilon}(b) 
\end{align*}
is bijective and the inverse of $\DcL_{\bfi, \epsilon}$ is given by 
$$
 \Gup \left( \DcL_{\bfi, \epsilon}^{-1}  (\bfa) \right) \equiv  f^{\up}_{\bfi, \epsilon} (\bfa) \mod qL(\infty).
$$

\begin{lemma}\label{Lem: two ld}	
Let $ \bfi = (i_1, i_2,  \ldots, i_\ell) \in \rR(w_0) $ and set $\bfi^\vee := (i_\ell^*, \ldots, i_2^*,  i_1^*) $.	
Then, for any $b \in B(\infty)$, we have 
\begin{align*}
	\DcL_{\bfi, -1}(b)	= \DcL_{\bfi^\vee, 1}(b).
\end{align*}
\end{lemma}	
\begin{proof}
Thanks to \eqref{Eq: rxw rot}, we have 
$$ 
(i_{k+1}^*, i_{k+2}^*,  \cdots, i_\ell^*, i_1, \ldots,  i_{k} ) \in \rR(w_0)   
$$
for any $k=0,1, \ldots, \ell$. 
Since 
$$
f_{i_k^*} = \lT_{ i_{k+1}^* }\lT_{i_{k+2}^*} \cdots \lT_{ i_{\ell}^* } \lT_{i_1} \cdots \lT_{i_{k-1}} (f_{i_k}),
$$
we have 
$$
\lTs_{ i_{\ell}^* } \cdots \lTs_{ i_{k+1}^* } (f_{i_k^*}) =  \lT_{i_1} \cdots \lT_{i_{k-1}} (f_{i_k}),
$$
which implies that 
$$
f_{\bfi^\vee, 1} (\beta_k^\vee) = f_{\bfi, -1} (\beta_k),
$$
where $\beta_k^\vee := s_{i_\ell^*} \cdots s_{i_{k+1}^*}(\alpha_k)$ and $\beta_k := s_{i_1} \cdots s_{i_{k-1}}(\alpha_k)$.
Therefore, since both of the PBW vectors are the same,  the assertion follows from \eqref{Eq: Gup PBW}.
\end{proof}

\begin{lemma} \label{Lem: T_i and ld}
Let $\bfi = (i_1, i_2, \ldots, i_\ell) \in \rR(w_0)$ and set $\bfj := (i_2, \ldots, i_\ell, i_1^*) $.	
Let $b \in B(\infty)$ and $t \in \Z_{\ge0}$.  
\bni
\item If we write $ \DcL_{\bfi, -1}(b) = ( m_{\ell}, m_{\ell-1}, \ldots, m_2, m_1) $, then  
$$
 \eps_{i_1}(b) = m_1, \qquad  \DcL_{\bfj, -1} \left( \tf_{i_1}^t \left( \tcTs_{i_1} (b) \right) \right)  = ( t,  m_{\ell}, m_{\ell-1}, \ldots, m_2).
$$
\item 
If we write $ \DcL_{\bfi, 1}(b) = (n_{1}, n_{2}, \ldots, n_{\ell-1}, n_\ell) $, then  
$$
\ep_{i_1}(b) = n_1, \qquad  \DcL_{\bfj, 1} \left( \tfss{t}_{i_1} \left( \tcT_{i_1} (b) \right) \right)  = (  n_2, \ldots, n_{\ell-1}, n_\ell, t).
$$
\ee
\end{lemma}	
\begin{proof}
(i)
It follows from Definition \ref{Def: LD} and the bijection $\DcL_{\bfi, \epsilon}$ that 
$$ 
\eps_{i_1}(b) = m_1 \quad \text{ and }\quad \DcL_{\bfi, -1} ( \tem_i b) = ( m_{\ell}, m_{\ell-1}, \ldots, m_2, 0).
$$
Since $ \tcTs_{i_{\ell}} \tcTs_{i_{\ell-1}} \cdots \tcTs_{i_2}\tcTs_{i_1}(b) 	= \one \in B(\infty)$, Definition \ref{Def: LD}  says that 
$$ 
  \DcL_{\bfj, -1} \left( \tcTs_{i_{1}} (b) \right) = ( 0, m_{\ell}, m_{\ell-1}, \ldots, m_2). 
$$
Since  Lemma \ref{Lem: two ld} tells us that $   \DcL_{\bfj, -1} \left( \tcTs_{i_{1}} b \right) =   \DcL_{\bfj^\vee, 1} \left( \tcTs_{i_{1}} b \right)$ and $\bfj^\vee = (i_1, i_\ell^*, \ldots, i_2^*) $, 
we obtain 
$$
\DcL_{\bfj, -1} \left( \tf_{i_1}^t \left( \tcTs_{i_1} (b) \right) \right)  = ( t,  m_{\ell}, m_{\ell-1}, \ldots, m_2).
$$

(ii) It can be proved by the same argument.
\end{proof}	

\vskip 2em

\section{Braid group action on $\hBi$} \label{Sec: braid group actions}

We keep the notations given in the previous section. 
Let $\hBi$ be the extended crystal of the crystal $B(\infty)$.
In this section, we show that there exists a braid group action on $\hBi$ arising from the Saito crystal reflections.

For $i\in I$ and $\cb = (b_k)_{k\in \Z} \in \hBi$, we define 
$$
\bR_i(\cb) = (b_k')_{k\in \Z}\qquad  \text{ and } \qquad \bRs_i(\cb) = (b_k'')_{k\in \Z}
$$ 
 by
\begin{align*}
b_k' := \tfss{ \ep_i(b_{k-1}) }_i  \left( \tcT_{i}(b_k) \right) \qquad \text{ and }\qquad 
b_k'' := \tf^{ \eps_i(b_{k+1}) }_i  \left( \tcTs_{i}(b_k) \right)
\end{align*}	
for any $k\in \Z$ respectively. Thus we have the maps 
$$
\bR_i : \hBi \longrightarrow \hBi \qquad \text{ and }\qquad 
\bRs_i : \hBi \longrightarrow \hBi.
$$
For any $\bfj = (j_1, \ldots, j_t) \in I^t$, we simply write 
\begin{align} \label{Eq: bR_bfi}
\bR_\bfj := \bR_{j_1} \bR_{j_2} \cdots \bR_{j_t} \quad \text{ and }\quad   \bRs_\bfj := \bRs_{j_1} \bRs_{j_2} \cdots \bRs_{j_t}. 
\end{align}

\begin{prop} \label{Prop: R and ld}
Let $\bfi = (i_1, i_2 \ldots, i_\ell) \in \rR(w_0)$ and set 
$
\bfj := ( i_{2}, i_{3}, \ldots, i_\ell, i_1^*).
$
Let $\cb =(b_k)_{k\in \Z} \in \hBi$ and write 
\begin{align*}
 \DcL_{\bfi, 1}(b_k) &= (n_{k, 1}, n_{k, 2}, \ldots, n_{k, \ell-1}, n_{k, \ell}), \\ 
 \DcL_{\bfi, -1}(b_k) &= (m_{k, 1}, m_{k, 2}, \ldots, m_{k, \ell-1}, m_{k, \ell}),
\end{align*}
for any $k\in \Z$.
\bni
\item  
If we write $\bR_{i_1}(\cb) = (b_k')_{k\in \Z} \in \hBi$, then  we have
\begin{align*}
	\DcL_{\bfj, 1} (b_k') = ( n_{k, 2}, n_{k, 3}, \ldots, n_{k, \ell}, n_{k-1, 1}) \qquad \text{ for any $k\in \Z$.}
\end{align*}
\item  If we write $\bRs_{i_1}(\cb) = (b_k'')_{k\in \Z} \in \hBi$, then we have
\begin{align*}
	\DcL_{\bfj, -1} (b_t'') = ( m_{k+1, \ell}, m_{k, 1}, m_{k, 2}, \ldots, m_{k, \ell-1} ) \qquad \text{for any $k \in \Z$.}
\end{align*}
\ee	
\end{prop}	
\begin{proof}
By Lemma \ref{Lem: T_i and ld} (ii), we have $\ep_{i_1} (b_{k-1}) = n_{k-1,1}$ and  
$$
\DcL_{\bfj, 1} (b_k') = \DcL_{\bfj, 1} \left ( \tfss{ \ep_{i_1}(b_{k-1}) }_{i_1}  \left( \tcT_{i_1}(b_k) \right) \right) = 
 ( n_{k, 2}, n_{k, 3}, \ldots, n_{k, \ell}, n_{k-1, 1}),
$$
which gives (i).

By the same argument, one can prove (ii). 	
\end{proof}

\begin{lemma} \label{Lem: R R*} 
	Let $i\in I$.
\bni 	
\item  $\bR_i$ and $\bRs_i$ are bijective.
\item $\bR_i$ and $\bRs_i$ are inverse to each other.  	
\item $\bR_i \circ \cd = \cd \circ   \bR_i $ and $\bRs_i \circ \cd = \cd \circ   \bRs_i $.
\item For any $\cb \in \hBi$, we have 
$$
 \hwt ( \bR_i (\cb) ) = s_i ( \hwt(\cb) ) \quad \text{ and }\quad  \hwt ( \bRs_i (\cb) ) = s_i ( \hwt(\cb) ).
$$

\ee
\end{lemma}	
\begin{proof}

We first show that $ \bRs_i \circ \bR_i = \id$.
We choose $\bfi = (i_1, i_2 \ldots, i_\ell) \in \rR(w_0)$ with $i=i_1$, and set
$$
\bfi^\vee := ( i_\ell^*, i_{\ell-1}^*, \ldots,   i_1^*), \qquad 
\bfj := ( i_{2}, i_{3}, \ldots, i_\ell, i_1^*), \qquad 
\bfj^\vee := ( i_{1}, i_{\ell}^*, \ldots, i_3^*, i_2^*).
$$
Let $\cb =(b_k)_{k\in \Z} \in \hBi$ and write 
$$
	\DcL_{\bfi, 1}(b_k) = (n_{k, 1}, n_{k, 2}, \ldots, n_{k, \ell-1}, n_{k, \ell})
$$
for any $k\in \Z$. We write 
$$
\bR_{i}(\cb) = (b_k')_{k\in \Z} \quad \text{ and } \quad  \bRs_{i} \circ \bR_{i}(\cb) = (b_k'')_{k\in \Z}.
$$
By Proposition \ref{Prop: R and ld} (i), we have 
\begin{align*}
	\DcL_{\bfj, 1} (b_k') = ( n_{k, 2}, n_{k, 3}, \ldots, n_{k, \ell}, n_{k-1, 1}) \qquad \text{ for any $k\in \Z$.}
\end{align*}
Since $\DcL_{\bfj^\vee, -1} (b_k') = \DcL_{\bfj, 1} (b_k') $ by Lemma \ref{Lem: two ld}, Proposition \ref{Prop: R and ld} (ii) tells us that 
$$
	\DcL_{\bfi^\vee, -1} (b_k'') = ( n_{k, 1}, n_{k, 2}, n_{k, 3}, \ldots, n_{k, \ell}) \qquad \text{ for any $k\in \Z$.}
$$
Therefore, Lemma \ref{Lem: two ld} says that $b_k = b_k ''$  for any $k \in \Z$, i.e., $ \bRs_i \circ \bR_i (\bfb) = \bfb $.

In the same manner, one can prove that $ \bR_i \circ \bRs_i = \id$.  Thus we have (i) and (ii).

(iii) It follows from the definitions of $ \bR_i$ and $ \bRs_i$.

(iv)
Let $ \cb = (b_k)_{k\in \Z} \in \hBi $ and write $ \bR_i(\cb) = (b_k')_{k\in \Z} $. For any $k\in \Z$, we have 
\begin{align*}
\wt_k(\bR_i(\cb)) &= (-1)^k \wt(b_k') = (-1)^k \wt \left ( \tfss{ \ep_i(b_{k-1}) }_i  \left( \tcT_{i}(b_k) \right) \right) \\
&= 	(-1)^k \wt \left ( \tfss{ \ep_i(b_{k-1}) }_i  \left( \cT_{i} \left( \te_i^{ \ep_i(b_k)}( b_k) \right) \right) \right) \\
&= (-1)^k ( s_i( \wt(b_k) + \ep_i( b_{k}) \al_i ) - \ep_i(b_{k-1}) \al_i ) \\
&= s_i( \wt_k( \cb )) +  \left((-1)^{k-1} \ep_i(b_{k-1})  - (-1)^{k} \ep_i(b_{k})  \right) \al_i,
\end{align*}	
which tells us that 
$$
\hwt( \bR_i(\cb)) = \sum_{k\in \Z} \wt_k( \bR_i(\cb))=
\sum_{k\in \Z} s_i ( \wt_k(  \cb)) =  s_i (\hwt(  \cb)).
$$

The case for $\bRs_i$ can be proved in the same manner.
\end{proof}	

\begin{remark} 
	Let $\star$ be the involution on $\hBi$ defined in \cite[Section 4]{KP22}, which is a counterpart of the involution $*$ on $B(\infty)$.
Using the fact that $ \cTs_i =  * \circ \cT_i \circ *$, one can prove that $ \bRs_i =  \star \circ \bR_i \circ \star $.
\end{remark}

Recall that $\cmA = (a_{i,j})_{i,j\in I}$ is a Cartan matrix of finite type.
For $i,j \in I$ with $i\ne j$, we set 
\begin{align} \label{Eq: m(i,j)}
m(i,j) := 
\begin{cases}
	2 & \text{ if } a_{i,j} a_{j, i} = 0,\\
	3 & \text{ if } a_{i,j} a_{j, i} = 1,\\
	4 & \text{ if } a_{i,j} a_{j, i} = 2,\\
	6 & \text{ if } a_{i,j} a_{j, i} = 3.
\end{cases}
\end{align}
Note that $m(i,j) = m(j,i)$ for any $i,j \in I$ with $i \ne j$.
 We denote by $\bg_\cmA$ the \emph{generalized braid group} (or \emph{Artin-Tits group}) defined by the generators $r_i$ ($i\in I$) and the following defining relations:
$$
\underbrace{r_{i}r_{j}r_{i}r_{j} \cdots}_{m(i,j) \text{ factors}} = \underbrace{r_{j}r_{i}r_{j}r_{i} \cdots}_{m(j,i) \text{ factors}} 
\qquad \text{ for $i,j\in I$ with $i\ne j$.}
$$ 
We set $\bg_\cmA^+$ to be the submonoid of $\bg_\cmA$ generated by $r_i$ ($i\in I$).
We call $ \bg_\cmA$ the braid group associated with $\cmA$. 
We simply write $\bg$ instead of $\bg_\cmA$ if no confusion arises.

\begin{theorem} \label{Thm: main1}
The bijections $\bR_i$ and $\bRs_i$ satisfy the braid  group relations for $ \bg$, i.e., for $i,j\in I$ with $i\ne j$,  
$$
\underbrace{\bR_{i}\bR_{j}\bR_{i}\bR_{j} \cdots}_{m(i,j) \text{ factors}} = \underbrace{\bR_{j}\bR_{i}\bR_{j}\bR_{i} \cdots}_{m(j,i) \text{ factors}} \quad \text{ and }\quad  
\underbrace{\bRs_{i}\bRs_{j}\bRs_{i}\bRs_{j} \cdots}_{m(i,j) \text{ factors}} = \underbrace{\bRs_{j}\bRs_{i}\bRs_{j}\bRs_{i} \cdots}_{m(j,i) \text{ factors}}.
$$	
\end{theorem}	
\begin{proof}
Let $\bfm (i,j) := ( \underbrace{ i,j,i,j, \ldots  }_{m(i,j) \ \text{times}} ) \in I^{m(i,j)}$. Then there exists a sequence $ \bfk \in I^{\ell- m(i,j)}$ such that 
$$
 \bfm (i,j) * \bfk \in \rR(w_0), 
$$
where $A*B$ is the concatenation of $A$ and $B$. 
Note that $\bfm (j,i) * \bfk \in \rR(w_0)$.
By \eqref{Eq: rxw rot}, the sequences $\bfk * \bfm (j^*, i^*)$ and $\bfk * \bfm (i^*, j^*)$ are also contained in $\rR(w_0)$. 
We set 
$$
\bfi_1 := \bfm (i,j) * \bfk, \quad \bfi_2 := \bfm (j,i) * \bfk, \quad 
\bfj_1 :=  \bfk * \bfm (i^*,j^*), \quad \bfj_2 :=  \bfk * \bfm (j^*,i^*). 
$$

Let $\cb =(b_k)_{k\in \Z} \in \hBi$, and for any $k\in \Z$ we write 
\begin{align*}
	\DcL_{\bfi_1, 1}(b_k) &= (x_{k, 1}, x_{k, 2}, \ldots, x_{k, \ell-1}, x_{k, \ell}), \\ 
	\DcL_{\bfi_2, 1}(b_k) &= (y_{k, 1}, y_{k, 2}, \ldots, y_{k, \ell-1}, y_{k, \ell}).
\end{align*}
 Let $m = m(i,j)$. By the construction, we have
\begin{align} \label{Eq: ld k}
x_{k,t} = y_{k,t} \qquad \text{ for any $k\in \Z$ and $ t \in \{m+1,m+2, \ldots, \ell  \}$.}
\end{align}
Moreover, since the transition map $T_{i,j}$ between two Lusztig's data was computed in \cite[Section 3]{BZ97} (see also \cite[Section 2.3]{SST18}), 
the first $m$ components of  $( x_{k,1}, x_{k,2}, \ldots, x_{k,\ell}  )$ and $( y_{k,1}, y_{k,2}, \ldots, y_{k,\ell}  )$ satisfy 
\begin{align} \label{Eq: ld braid}
T_{i,j} ( x_{k,1}, x_{k,2}, \ldots, x_{k,m}  ) = (y_{k,1}, y_{k,2}, \ldots, y_{k,m}) \quad \text{ for any $k\in \Z$.}  
\end{align}

We write $\bR_{\bfm(j,i)} ( \bfb) = (b_k')_{k\in \Z}$ and $\bR_{\bfm(i,j)} ( \bfb) = (b_k'')_{k\in \Z}$. By Proposition \ref{Prop: R and ld}, we have
\begin{align*}
	\DcL_{\bfj_1, 1} (b_k') &= ( x_{k, m+1}, x_{k, m+2}, \ldots, x_{k, \ell}, x_{k-1, 1}, \ldots, x_{k-1, m}), \\
	\DcL_{\bfj_2, 1} (b_k'') &= ( y_{k, m+1}, y_{k, m+2}, \ldots, y_{k, \ell}, y_{k-1, 1}, \ldots, y_{k-1, m})
\end{align*}
for any $k\in \Z$.
Since $a_{i,j} = a_{i^*, j^*}$,  the transition map $T_{i,j}$ coincides with $T_{i^*, j^*}$. By \eqref{Eq: ld k} and \eqref{Eq: ld braid},  the transition map $T_{i^*,j^*}$ sends $\DcL_{\bfj_1, 1} (b_k') $ to $\DcL_{\bfj_2, 1} (b_k'') $. We thus conclude that $b_k' = b_k'' $ for any $k\in \Z$, i.e., 
$$ 
\bR_{\bfm(j,i)} ( \bfb) = \bR_{\bfm(i,j)} ( \bfb).
$$ 
The relation for $\bR_i^*$	follows from Lemma \ref{Lem: R R*}.
\end{proof}

Thanks to Theorem \ref{Thm: main1}, the braid group $\bg$ acts on the extended crystal $\hBi$ as follows:
\begin{align*}
r_i \cdot \cb := \bR_i(b) \quad \text{ and }\quad 	r_i^{-1} \cdot \cb := \bR_i^*(b) \qquad \text{ for $i\in I$ and $\cb \in \hBi$.}
\end{align*}	
For any $w = r_{i_1}^{\epsilon_1} r_{i_2}^{\epsilon_2} \cdots r_{i_t}^{\epsilon_t} \in \bg$ with $\epsilon_k \in \{ -1,1 \}$, we define 
\begin{align} \label{Eq: bR_w}
\bR_w :=  R_{1}  R_{2} \cdots  R_{t},
\end{align}
where $R_{k} := 
\begin{cases}
	\bR_{i_k} & \text{ if } \epsilon_k = 1, \\
\bRs_{i_k} & \text{ if } \epsilon_k = -1.
\end{cases}
$

\begin{remark} \label{Rmk: tran}
 By Lemma \ref{Lem: R R*} (iv),  the set $\hBi_0$ of weight $0$ is invariant under the action of $\bg$. 	
 Thus the action of $\bg$ on $\hBi$ is not transitive.
 In the case where $\cmA$ is of type $A_1$, i.e. $\g = \mathfrak{sl}_2$,  it is easy to prove that $\hBi$ is faithful as a $\bg$-set. It is conjectured that $ \hBi$ is faithful as a $\bg$-set for any Cartan matrix $\cmA$.
  
\end{remark}	

Let us recall the involution $\zeta$ defined in \eqref{Eq: i->i*}. Since the involution $\zeta$ gives an involution on $B(\infty)$, we have the induced involution on the extended crystal $\hBi$,
i.e., 
$$
\zeta (\bfb) := ( \zeta(b_k))_{k\in \Z} \qquad \text{ for any $\bfb  = (b_k)_{k\in \Z} \in \hBi$.}
$$

\begin{theorem} \
\bni
\item For any $i\in I$, we have $ \bR_i \circ \zeta = \zeta \circ \bR_{\zeta(i)}   $ and $ \bRs_i \circ \zeta = \zeta \circ \bRs_{\zeta(i)} $.		
\item For any $\bfi \in \rR(w_0)$, we have 
$$
\bR_\bfi = \cd \circ \zeta \qquad \text{ and } \qquad \bRs_\bfi = \cd^{-1} \circ \zeta.
$$ 
In particular, unless the Cartan matrix $\cmA$ is of type $A_n$ ($n\in \Z_{>1}$), $D_{n}$ ($n$ is odd) or $E_6$, we have
$$
\bR_\bfi = \cd   \qquad \text{ and } \qquad \bRs_\bfi = \cd^{-1}.
$$ 
\ee
\end{theorem}
\begin{proof}
	
(i)	 Since  
\begin{align} \label{Eq: tf zeta}
	\tf_i \left(\zeta (b) \right) = \zeta \left(  \tf_{\zeta(i)}  (b) \right)  \quad \text{ and } \quad \te_i \left( \zeta (b) \right) = \zeta \left(  \te_{\zeta(i)}  (b) \right)
\end{align}	
 for any $i\in I$ and  $b \in B(\infty)$, we have 
\begin{align} \label{Eq: T zeta}
	\cT_i \left(\zeta (b') \right) = \zeta \left(  \cT_{\zeta(i)}  (b')  \right) \quad \text{ and } \quad \cTs_i \left(\zeta (b'') \right) = \zeta \left(   \cTs_{\zeta(i)}  (b'') \right)
\end{align}	
for any $b' \in {_{\zeta(i)}B(\infty)}$ and $b'' \in B(\infty)_{\zeta(i)}$.
Thus (i) follows from the definitions of $\bR_i$ and $\bRs_i$.

(ii) 
Let $\bfi = (i_1, i_2, \ldots, i_\ell) \in \rR(w_0)$ and set 
$$
 \bfj := (i_{\ell}, i_{\ell-1}, \ldots, i_1  ), \qquad  \bfj^* := (i_{\ell}^*, i_{\ell-1}^*, \ldots, i_1^*).
$$ 
By \eqref{Eq: tf zeta} and \eqref{Eq: T zeta}, we have 
$$
f_{\bfj^*, \epsilon} (\beta_k^*) =  \zeta( f_{\bfj, \epsilon} (\beta_k)),
$$
where $\beta_k^* := s_{i_\ell^*} \cdots s_{i_{\ell-k+2}^*}(\alpha_{i^*_{\ell-k+1}})$ and $\beta_k := s_{i_\ell} \cdots s_{i_{\ell - k+2}}(\alpha_{i_{\ell-k+1}})$.
Thus, for any $ b\in B(\infty)$, if we write $  \DcL_{\bfj, \epsilon}(b) = (n_1, n_2, \ldots, n_{\ell-1}, n_{ \ell}) $, then we have  
\begin{align} \label{Eq: j* ld}
\DcL_{\bfj^*, \epsilon}( \zeta(b)) = (n_1, n_2, \ldots, n_{\ell-1}, n_{ \ell}). 
\end{align}

Let $\bfb = (b_k)_{ k\in \Z} \in \hBi $, and write
$ \DcL_{\bfj, 1}(b_k) = (n_{k, 1}, n_{k, 2}, \ldots, n_{k, \ell-1}, n_{k, \ell})
$
for $k\in \Z$. Let $ \bR_{\bfi} (\bfb) = (b_k')_{k\in \Z}$. By Proposition \ref{Prop: R and ld} and \eqref{Eq: j* ld}, we have 
\begin{align*}
\DcL_{\bfj^*, 1}(b_k')= (n_{k-1, 1}, n_{k-1, 2}, \ldots, n_{k-1, \ell-1}, n_{k-1, \ell}) = \DcL_{\bfj^*, 1}( \zeta ( b_{k-1})),
\end{align*}	
which implies that $\bR_\bfi = \cd \circ \zeta $.  In the same manner, one can prove $ \bRs_\bfi = \cd^{-1} \circ \zeta$. 

The last identifies follow from the fact that $\zeta$ is the identity unless $\cmA$ is of type $A_n$ ($n\in \Z_{>1}$), $D_{n}$ ($n$ is odd) or $E_6$.
\end{proof}	

\begin{example} \label{Ex: braid}
We keep the notations given in Example \ref{Ex: multisegment} and Example \ref{Ex: multi_EC}.	

Let $n=2$. In this case, $\rR(w_0) = \{ (1,2,1), (2,1,2) \}$ and the segments of $\MS_2$ are $[1]$, $[12]$, and $[2]$.
\bni
\item 
Let $\hbfm = (  \ldots, \emptyset, \emptyset, \underline{ [2]}, \emptyset, \emptyset, \ldots ) \in \hMS_2$, where the underlined element is located at the 0-position.
Note that $\tF_{2,0} (\one)= \hbfm$.
By direct computations, we have
\begin{align*}
\bR_1(\hbfm) &= (  \ldots, \emptyset, \emptyset, \underline{ [12]}, \emptyset, \emptyset, \ldots ), \\
\bR_2\bR_1(\hbfm) &= (  \ldots, \emptyset, \emptyset, \underline{ [1]}, \emptyset, \emptyset, \ldots ), \\
\bR_1\bR_2\bR_1(\hbfm) &= (  \ldots, \emptyset, [1], \underline{ \emptyset }, \emptyset, \emptyset, \ldots ),
\end{align*}
and 
\begin{align*}
	\bR_2(\hbfm) &= (  \ldots, \emptyset, [2], \underline{ \emptyset }, \emptyset, \emptyset, \ldots ), \\
	\bR_1\bR_2(\hbfm) &= (  \ldots, \emptyset, [12], \underline{ \emptyset }, \emptyset, \emptyset, \ldots ), \\
	\bR_2\bR_1\bR_2(\hbfm) &= (  \ldots, \emptyset, [1], \underline{ \emptyset }, \emptyset, \emptyset, \ldots ).
\end{align*}
Thus we have 
$$
\bR_1\bR_2\bR_1(\hbfm) = \bR_2\bR_1\bR_2(\hbfm) = \cd \circ \zeta (\hbfm).
$$

\item 
Let $\hbfm = (  \ldots, \emptyset, \emptyset, \bfm_1, \underline{ \bfm_0 }, \bfm_{-1}, \emptyset, \emptyset, \ldots ) \in \hMS_2$, where 
$\bfm_1 = 2  [2] + 3[12] + 4[1]$, $\bfm_0 =  [2] + [12] + 2[1]$  and $\bfm_{-1} = 3  [2] + 2[12] + [1]$.
Then we have 
\begin{align*}
\ep_1 (\bfm_1) = 5, \qquad 
\ep_1 (\bfm_0) = 2, \qquad 
\ep_1 (\bfm_{-1}) = 2,
\end{align*}
and 
\begin{align*}
\tem_1 (\bfm_1) = 5[2] + 2[1], \quad 
\tem_1 (\bfm_0) = 2[2]+[1], \quad 
\tem_1 (\bfm_{-1}) = 5[2]+[1]. 
\end{align*}
Thus, we have 
\begin{align*}
 \tcT_1(\bfm_1) = 2[2] + 3[12] , \quad 
\tcT_1 (\bfm_0) = [2]+[12], \quad 
\tcT_1 (\bfm_{-1}) = [2]+4[12], 
\end{align*}
which implies that 
$\bR_1 (\hbfm) = (  \ldots, \emptyset, \bfm_2', \bfm_1', \underline{ \bfm_0' }, \bfm_{-1}', \emptyset, \emptyset, \ldots )$, where 
\begin{align*}
\bfm_2' &= 5 [1], \\ 
\bfm_1' &= 2[2] + 3[12] + 2 [1], \\
\bfm_0' &= [2]+[12] + 2[1], \\
\bfm_{-1}' &= [2]+4[12].
\end{align*}

\ee

\end{example}

\vskip 2em

\section{Similarity} \label{Sec: similarity}
In this section, we apply the result of \cite{K96} to the extended crystals. 

Let $J$ be a finite set and let $\vt: I \twoheadrightarrow J$ be a surjective map. We take positive integers $m_i$ for each $i\in I$, and set 
$$ 
\tal_j := \sum_{i \in \vt^{-1}(j)} m_i \al_i,
$$ 
where $\vt^{-1}(j) = \{ i\in I \mid \vt(i)=j \}$.
We define $ \twlP$ to be the subset of $\wlP$ consisting of $\la \in \wlP$ such that, for any $j\in J$, the value $ \frac{\langle h_i, \la \rangle}{m_i}$ is an integer and does not depend on the choice of $i \in\vt^{-1}(j)$. 
For $i\in I$, we denote by $ \tch_i \in \twlP^\vee := \Hom_\Z(\twlP, \Z)$ the element satisfying $ \langle \tch_j, \la \rangle =  \frac{\langle h_i, \la \rangle}{m_i}$ for $i\in \vartheta^{-1}(j)$ and $\la \in \twlP$.
We assume that 
\bna
\item $\langle h_i, \alpha_{i'} \rangle =0$ for $i,i'\in I$ such that $ \vartheta(i) = \vartheta(i')$ and $i\ne i'$,
\item $\tal_j$ is contained in $ \twlP$ for any $j\in J$.
\ee

We define the matrix $\cmA_\vt = (c_{i,j})_{i,j\in J}$ by 
$$
c_{i,j }:= \langle \tch_i, \tal_j \rangle\qquad \text{for any } i,j \in J.
$$
Note that $ \cmA_\vartheta$ is a generalized Cartan matrix. Let $B(\infty)$ and $B_\vt(\infty)$ be the crystals of the negative halfs of the quantum groups associated with $\cmA$ and $\cmA_\vt$ respectively.
By \cite[Section 5]{K96} (see \cite[Proposition 3.2]{NaSa03} and \cite[Theorem 2.3.1]{NaSa05} for the $*$-operator), we have the following theorem.

\begin{theorem} [{\cite[Section 5]{K96}}] 
There exists a unique injective map $\nP_\vt: B_\vt(\infty) \longrightarrow B(\infty)$ such that $ \nP_\vt(\hv) = \hv $ and, for any $j\in J$,   
\begin{align*}
	\nP_\vt( \te_j(b)) &=  \left(\prod_{i \in \vt^{-1}(j)  } \te_i^{m_i} \right) \nP_\vt(b), \qquad 
	\nP_\vt( \tf_j(b)) =  \left(\prod_{i \in \vt^{-1}(j)  } \tf_i^{m_i} \right) \nP_\vt(b), \\
	\nP_\vt( \tes_j(b)) &=  \left(\prod_{i \in \vt^{-1}(j)  } \tess{m_i}_i \right) \nP_\vt(b), \qquad 
	\nP_\vt( \tfs_j(b)) =  \left(\prod_{i \in \vt^{-1}(j)  } \tfss{m_i}_i \right) \nP_\vt(b).
\end{align*}
\end{theorem}

The above theorem can be easily extended to the extended crystals.

\begin{corollary} \label{Cor: folding for hBi}
There exists a unique injective map $\hP_\vt: \hBvi \longrightarrow \hBi$ such that  $ \hP_\vt(\one) = \one $ and, for any $(j,k) \in \cJ$,    
\begin{align*} 
	\hP_\vt( \tE_{j,k}(\cb)) &=  \left(\prod_{i \in \vt^{-1}(j)  } \tE_{i,k}^{m_i} \right) \hP_\vt(\cb), \qquad 
	\hP_\vt( \tF_{j,k}(\cb)) =  \left(\prod_{i \in \vt^{-1}(j)  } \tF_{i,k}^{m_i} \right) \hP_\vt(\cb).
\end{align*}
\end{corollary}

An automorphism $\sigma$ of a Dynkin diagram gives a desired example for $\vt$ (see \cite[Section 5]{K96} and  \cite[Section 14.4]{Lu93}). 
We illustrate such examples as follows. We take $m_i = 1$ for any $i\in I$. 
\begin{equation} \label{Eq: Folding1}
\begin{aligned} 
	\small
& \xymatrix@R=0.6em@C=0.9em{
 &\circ \ar@{-}[r] \ar@{.}[dd] & \cdots \ar@{-}[r]  & \circ \ar@{-}[r] \ar@{.}[dd] & \circ \ar@{-}[dr]  \ar@{.}[dd] \\
 A_{2n-1} : & &&&& \circ  \ar@{.>}[ddd] \\
&  \circ \ar@{-}[r] \ar@{.>}[dd] & \cdots \ar@{-}[r] & \circ \ar@{-}[r] \ar@{.>}[dd] & \circ \ar@{-}[ru] \ar@{.>}[dd]  \\
&    &&&& \\
B_{n} :&  \circ \ar@{-}[r] & \cdots \ar@{-}[r] & \circ \ar@{-}[r] & \circ \ar@{=>}[r] & \circ 
}
\quad  \qquad
\xymatrix@R=0.6em@C=0.9em{
	&  \circ \ar@{-}[r]  \ar@{.}[dd]  & \circ  \ar@{-}[rd] \ar@{.}[dd] &&   \\
	E_6 : &   && \circ \ar@{-}[r] \ar@{.>}[ddd] & \circ \ar@{.>}[ddd]  \\
	& \circ \ar@{-}[r] \ar@{.>}[dd] & \circ \ar@{-}[ur] \ar@{.>}[dd] &&    \\
	&    &&&& \\
	F_4 :&  \circ \ar@{-}[r] & \circ \ar@{=>}[r] & \circ \ar@{-}[r] & \circ  
}
\\
& \xymatrix@R=0.6em@C=0.9em{
	& &&&& \circ  \ar@{.}[dd] \\
	D_{n+1} : & \circ \ar@{-}[r] \ar@{.>}[ddd] & \cdots \ar@{-}[r] & \circ  \ar@{-}[r] \ar@{.>}[ddd] & \circ  \ar@{-}[ru] \ar@{-}[rd]  \ar@{.>}[ddd] &  \\
	& &&&& \circ  \ar@{.>}[dd]   \\
	&    &&&& \\
	C_{n} :&  \circ \ar@{-}[r] & \cdots \ar@{-}[r] & \circ \ar@{-}[r] & \circ \ar@{<=}[r] & \circ  
}
\qquad \qquad 
\xymatrix@R=0.6em@C=0.9em{
	 &\circ  \ar@{-}[rd]  \ar@{.}[d]   \\
	D_{4} : &  \circ \ar@{-}[r] \ar@{.}[d] & \circ \ar@{.>}[ddd]  \\
	 & \circ  \ar@{-}[ru] \ar@{.>}[dd]    \\
	&    &&&& \\
	G_{2} :&  \circ    \ar@3{->}[r] & \circ 
}
\end{aligned}	 	
\end{equation}
In this case, we simply write $ \hBsi$, $\hP_\sigma$, and $\cmA_\sigma$ for $ \hBvi$, $\hP_\vt$, and $\cmA_\vt$ respectively, and  set $ \orb_\sigma(j) := \theta^{-1}(j)$ for any $j\in J$.
Note that $\zeta$ defined in \eqref{Eq: i->i*} coincides with $\sigma$ if $\zeta$ is not the identity.

There are another examples for $\vt$ in \cite{K96}. 
\begin{align*}
	\small
& \xymatrix@R=0.6em@C=0.9em{
	&\circ \ar@{-}[r] \ar@{.}[dd] & \cdots \ar@{-}[r]  & \circ \ar@{-}[r] \ar@{.}[dd] & \circ \ar@{-}[dr]  \ar@{.}[dd] \\
	A_{2n-1} : & &&&& \bullet  \ar@{.>}[ddd] \\
	&  \circ \ar@{-}[r] \ar@{.>}[dd] & \cdots \ar@{-}[r] & \circ \ar@{-}[r] \ar@{.>}[dd] & \circ \ar@{-}[ru] \ar@{.>}[dd]  \\
	&    &&&& \\
	C_{n} :&  \circ \ar@{-}[r] & \cdots \ar@{-}[r] & \circ \ar@{-}[r] & \circ \ar@{<=}[r] & \circ 
}
\quad  \qquad
\xymatrix@R=0.6em@C=0.9em{
	&\circ  \ar@{-}[r]  \ar@{.}[d] & \circ \ar@{-}[ld] \ar@{.}[d]    \\
	A_{5} : &  \bullet \ar@{-}[r] \ar@{.}[d] & \circ \ar@{.>}[ddd]  \\
	& \circ  \ar@{-}[ru] \ar@{.>}[dd]    \\
	&    &&&& \\
	G_{2} :&  \circ    \ar@3{->}[r] & \circ 
}
\end{align*}	
Here, for $i\in I$, we set 
$m_i :=  
\begin{cases}
	1 & \text{ if the vertex at $i$ is $\circ$},\\
	2 & \text{ if the vertex at $i$ is $\bullet$}.
\end{cases}
$

\smallskip
Let $\sigma$ be an automorphism given in  \eqref{Eq: Folding1}.
Since  $\sigma$ acts on $B(\infty)$, 
we obtain the induced automorphism of $ \hBi$, i.e., 
$$
\sigma(\cb) := (\sigma(b_k))_{k\in \Z} \qquad \text{ for any } \cb = (b_k)_{k\in \Z} \in \hBi.
$$
Let $\bg = \langle r_i \mid i\in I \rangle$ be the Braid group associated with $\cmA$ and let $ \bg_\sigma =  \langle r_j' \mid j \in J \rangle$   be the Braid group associated with $\cmA_\sigma$. 
Then $\bg_\sigma$ is embedded in $\bg$ via the injection 
$$
 \iota_\sigma: \bg_\sigma \rightarrowtail \bg  
$$
 defined by $\iota_\sigma( r_j') = \prod_{ i\in \orb_\sigma(j)} r_i$ for $j\in J$ (see \cite{Crisp96}).

One can prove the following proposition by using Lusztig's result (\cite[Section 14.4]{Lu93}) and Corollary \ref{Cor: folding for hBi}.
\begin{prop} \label{Prop: folding}
Let $\sigma$ be an automorphism given in \eqref{Eq: Folding1}.  
\bni
\item Let $ \hBi^\sigma$ be the set of fixed points of $\hBi$ under the action $\sigma$. For any $(j,k)\in \cJ$, we define 
$$
\fE_{j,k}^\sigma := \prod_{i \in \orb_\sigma(j)  } \tE_{i,k} \quad \text{ and }\quad 
\fF_{j,k}^\sigma := \prod_{i \in \orb_\sigma(j) } \tF_{i,k}.
$$ 
 Then the map $\hP_\sigma $ given in Corollary \ref{Cor: folding for hBi} is compatible with the extended crystal operators, i.e., 
 $$
 \hP_\sigma \left( \tE_{j,k} (\bfb) \right) = \fE_{j,k}^\sigma \left( \hP_\sigma (  \bfb) \right) \quad \text{ and } \quad   \hP_\sigma \left( \tF_{j,k} (\bfb) \right) = \fF_{j,k}^\sigma \left( \hP_\sigma (  \bfb) \right)
 $$ 
 for any $(j,k)\in \cJ$ and $ \bfb \in \hBsi$, 
 which provides an isomorphism of extended crystals 
$$
\hP_\sigma :  \hBsi \buildrel \sim \over \longrightarrow \hBi^\sigma.
$$

\item Let $\bg = \langle r_i \mid i\in I \rangle$ be the Braid group associated with $\cmA$ and let $ \bg_\sigma =  \langle r_j' \mid j \in J \rangle$   be the Braid group associated with $\cmA_\sigma$. 
For $j\in J$, we define 
\begin{align} \label{Eq: folding braid generators}
\fdr^\sigma_{j} := \prod_{i \in \orb_\sigma(j)  } r_i \in \bg. 
\end{align}
Then the isomorphism $\hP_\sigma $ is compatible with the braid group actions, i.e., 
$$
\hP_\sigma \left(r_j'(\cb)\right) = \fdr^\sigma_j \left(\hP_\sigma (\cb) \right) \qquad \text{ for any }j\in J.
$$
\ee
\end{prop}

\vskip 2em

\section{Hernandez-Leclerc categories} \label{Sec: HL cat}

We keep the notation given in Section \ref{subsec: HL}. 
Let $\ddD := \{ \Rt_i \}_{i\in \If} $ be a complete duality datum of $\catCO$ and we write $\cmAf = (a^{\fin}_{i,j})_{i,j\in \If}$.
It is announced in \cite{KKOP21A} that 
\bni
\item  there exists an action of the braid
group $\bg_\cmAf$ associated with the Cartan matrix $\cmAf$ on the quantum Grothendieck ring of the Hernandez-Leclerc category $\catCO$, 
\item there exists a family of monoidal autofunctors $\mathscr{S}_i$ on the localization  $\T_N$ such that the autofunctors induce the braid group action in (i) at the Grothendieck ring level.
\ee
Note that  $\T_N$ can be understood as a graded version of the Hernandez-Leclerc category $\catCO$ for affine type $A_{N-1}$.   
It is also conjectured in \cite[Section 5]{KKOP21A} that, for an arbitrary quantum affine algebra $U_q'(\g)$, there exist monoidal autofunctors on the Hernandez-Leclerc category $\catCO$ with the same properties as the autofunctors $\mathscr{S}_i$ on $\T_N$. In this section, we assume that this conjecture is true. Namely, we assume that the following conjecture holds.

\Conj \label{Conj: R_i}
There exist exact monoidal autofunctors $\{ \fR_i \}_{i\in \If}$ on $\catCO$ satisfying the following properties: for any $i,j \in \If$, 
\bna
\item $\fR_i$ sends simple modules to simple modules,
\item $ \fR_i(\Rt_i) \simeq \dual(\Rt_i)$, 
\item  $ \fR_i \circ \dual \simeq \dual \circ  \fR_i$, 
\item if $ a^{\fin}_{i,j} = 0$, then $\fR_i \circ \fR_j \simeq \fR_j \circ \fR_i $, 
\item if $ a^{\fin}_{i,j} = -1$, then $\fR_i \circ \fR_j \circ \fR_i \simeq \fR_j \circ \fR_i \circ \fR_j $, 
\item \label{conj: diag} at the Grothendieck ring level, the following diagram commutes,
$$
\xymatrix{
K({^{i}R_{\gf}}\gmod ) \ar[rr]^{ \  \quad  [\F_\ddD]} \ar[d]_{[\fT_i]}  \ar[drr]^{[\F_{\Refn_i (\ddD)}]}    && K(\catCO) \ar[d]^{[\fR_i]} \\
K({_{i}R_{\gf}}\gmod ) \ar[rr]^{ \  \quad  [\F_\ddD]} && K(\catCO),
}
$$
\ee
where we use the following notations:
\begin{itemize}
\item  $ {^{i}R_{\gf}}\gmod$ (resp.\ $ {_{i}R_{\gf}}\gmod$) is the full subcategory of $R_{\gf}\gmod$ consisting of graded modules $M$ with $E_i M = 0$ (resp.\ $E_i^* M=0$),
\item 
$\fT_i : {^{i}R_{\gf}}\gmod \rightarrow {_{i}R_{\gf}}\gmod$ is the reflection functor due to S. Kato \cite{Kato14}, 
\item $\Refn_i (\ddD)$ is the duality datum obtained from $\ddD$ by applying the reflection $\Refn_i$ introduced in \cite[Section 5.3]{KKOP20A}. 
\end{itemize}
\enconj
The above properties (a)-(f) of Conjecture \ref{Conj: R_i} come from \cite[Theorem 2.4, Theorem 3.1, Proposition 3.2 and Section 5]{KKOP21A}. 
Note that $ \fT_i$ is the same as $\mathbb{T}_i^{-1}$ in \cite[Theorem 2.4]{KKOP21A} and 
that $\fT_i$ is compatible with the Saito crystal reflection $ \cT_i$ (see \cite[Theorem 3.6]{Kato14}). 
Since $ \fT_i$, $ \fR_i$ and $\F_\ddD$ send simple modules to simple modules, it follows from \ref{conj: diag} that 
\begin{align*}
\F_\ddD \circ \fT_i (M) \simeq \fR_i \circ \F_\ddD  (M) \qquad \text{ for a simple module $M$ in ${^{i}R_{\gf}}\gmod$.}
\end{align*}
In particular, if we write $ \F_\ddD(M) \simeq  \cL_\ddD(b) $ for some $b\in  {_{i}}B_{\gf}(\infty)$ (see \eqref{Eq: Fd LD} for the definition of $\cL_\ddD$), then we have 
\begin{align} \label{Eq: comm}
	\cL_\ddD (\cT_i ( b  )) \simeq \fR_i  (\cL_\ddD(b)),
\end{align}
where $ \cT_i$ is the Saito crystal reflection. 

\begin{prop} \label{Prop: braid for CgO}
We assume that Conjecture \ref{Conj: R_i} is true. Then we have 
$$
\Phi_\ddD ( \bR_i (\cb) ) = \fR_i ( \Phi_\ddD(\cb) )  \qquad \text{ for any $\cb \in \cBg[\gf]$ and $i\in \If$,}
$$
where $\Phi_\ddD \col  \cBg[\gf]  \buildrel \sim \over \longrightarrow \sBD(\g)$ is the extended crystal isomorphism given in Section \ref{subsec: HL}, and $\bR_i$ is the braid group action on $\cBg[\gf] $ defined in Section \ref{Sec: braid group actions}.
\end{prop}
\begin{proof}
Let $\cb = (b_k)_{k\in \Z} \in \cBg[\gf]$ and let $L := \Phi_\ddD(\cb)$.  
By the definition of $\Phi_\ddD$, we have 
\begin{align*}
	\Phi_\ddD(\cb) = \hd \left( \bigotimes_{k=+\infty }^{-\infty} \dual^k L_k \right) =
	 \hd ( \cdots \tens \dual^2 L_2 \tens \dual L_1 \tens  L_0 \tens \dual^{-1} L_{-1} \tens \cdots ), 
\end{align*}
where $L_k := \cL_\ddD(b_k) \in \catCD$ for $ k\in \Z$ (see \cite[Lemma 3.2]{KP22}). 
For any $k\in \Z$, let 
$$
 a_k := \ep_i(b_k) \quad\text{ and }\quad   L_k' := \cL_\ddD( \te_i^{a_k} ( b_k)).
 $$
We then have $L_k \simeq \Rt_i^{\tens a_i} \hconv L_k'$  and, by \eqref{Eq: comm}, 
\begin{align} \label{Eq: FT}
\fR_i  (L_k') \simeq  \cL_\ddD (\cT_i ( \te_i^{a_k} ( b_k)  )) = \cL_\ddD (\tcT_i (  b_k  )). 
\end{align}

By \cite[Lemma 2.6, Lemma 3.2 and Lemma 5.7]{KP22}, \eqref{Eq: FT} and the properties of $\fR_i$, we have 
\begin{align*}
\fR_i(L) &=  \fR_i \left( \hd \left( \bigotimes_{k=+\infty }^{-\infty} \dual^k  \Rt_i^{\tens a_k} \tens \dual^k L_k' \right) \right) \\
&\simeq     \hd \left( \bigotimes_{k=+\infty }^{-\infty} \dual^k \left( \fR_i( \Rt_i )^{\tens a_k} \right) \tens \dual^k \left(\fR_i( L_k') \right) \right)  \\ 
&\simeq     \hd \left( \bigotimes_{k=+\infty }^{-\infty} \dual^{k+1} \left(   \Rt_i ^{\tens a_k} \right) \tens \dual^k \left( \cL_\ddD (\tcT_i (   b_k  )) \right) \right)   \\
&\simeq     \hd \left( \bigotimes_{k=+\infty }^{-\infty}  \dual^k \left( \cL_\ddD (\tcT_i  ( b_k) ) \right) \tens \dual^{k} \left(   \Rt_i ^{\tens a_{k-1}} \right) \right)  \\
&\simeq     \hd \left( \bigotimes_{k=+\infty }^{-\infty}  \dual^k \left( \cL_\ddD (\tcT_i   ( b_k)  )  \htens    \Rt_i ^{\tens a_{k-1}} \right) \right)  \\
&\simeq     \hd \left( \bigotimes_{k=+\infty }^{-\infty}  \dual^k \left( \cL_\ddD (  \tfss{ a_{k-1} }_i  \tcT_i ( b_k  ))   \right) \right)  \\
&\simeq  \Phi_\ddD ( \bR_i (\cb) ).
\end{align*}		
\end{proof}

\begin{example}
We keep all notations given in Example \ref{Ex: ECI} and Example \ref{Ex: braid}. 
Recall the complete duality datum  $\ddD = \{  \Rt_1, \Rt_2   \} $, where $\Rt_1 := V(\varpi_1) $ and $\Rt_2 := V(\varpi_1)_{(-q)^2} $.
We assume that Conjecture \ref{Conj: R_i} is true.

\bni
\item 
We set 
$$
\bfi = (1,2,1,2,1,2, \ldots ),
$$
and write $\bfi = (i_k)_{k \in \Z_{>0}}$. Mimicking the construction to make PBW vectors in a quantum group, we define 
$$
V_k := \fR_{i_1}  \fR_{i_2} \cdots \fR_{i_{k-1}} (\Rt_{i_k}) \qquad \text{ for any $k \in \Z_{>0}$.}
$$ 
By Proposition \ref{Prop: braid for CgO}, we have 
\begin{align} \label{eq: P}
	P := \left\{  V_k \mid k\in \Z_{>0} \right\} = 	\left\{  \Phi_\ddD (\bR_{i_1}  \bR_{i_2} \cdots \bR_{i_{k-1}} (  \tF_{i_k,0} (\one))  \mid k\in \Z_{>0} \right\}.
\end{align}	
Let
$\hbfm_2 = ( \ldots,  \emptyset, \underline{[2]}, \emptyset, \ldots   )$,  
$\hbfm_{12} = ( \ldots,  \emptyset, \underline{[12]}, \emptyset, \ldots   ),$ and  
$\hbfm_1 = ( \ldots,  \emptyset, \underline{[1]}, \emptyset, \ldots )$.
Note that 
$$
\Phi_\ddD ( \hbfm_2) = V(\varpi_1)_{(-q)^2}, \quad \Phi_\ddD ( \hbfm_{12}) = V(\varpi_2)_{-q}, \quad 
\Phi_\ddD ( \hbfm_1) = V(\varpi_1).
$$
By the same argument given in Example \ref{Ex: braid} (i), it follows from \eqref{eq: P} that
\begin{align*}
P &= \left\{  \Phi_\ddD (\cdual^t ( \hbfm_2) ), \ \Phi_\ddD (\cdual^t ( \hbfm_{12}) ), \ \Phi_\ddD (\cdual^t ( \hbfm_1) )  \mid t \in \Z_{\ge0} \right\} \\
&= 	\left\{  V(\varpi_1)_{(-q)^a},\  V(\varpi_2)_{(-q)^b} \mid a \in 2\Z_{\ge0}, \ b \in 2\Z_{\ge0}+1 \right\}.
\end{align*}	
We remark that the set $P$ can be viewed as the set of fundamental modules contained in the Hernandez-Leclerc category $\catC_\g^-$ introduced in \cite{HL16}. 
From the viewpoint of the PBW theory given in \cite{KKOP20A},  the set $P$ is also understood as the set of \emph{affine cuspidal modules} in $\catC_\g^-$ corresponding to $\ddD$ and the reduced expression $s_1s_2s_1$.  

\item We recall the element $\hbfm$ given in Example  \ref{Ex: braid} (ii).  Let $\la := \gamma(\hbfm) \in \crBB{2}$, where $\gamma$ is given in \eqref{eq: gamma}.
 One can write 
$
\la = \la_1 + \la_0 + \la_{-1}, 
$ 
where 
\begin{align*}
\la_1 &:= \gamma_1(\hbfm_1) = 2  (2,5) + 3 (1,4) + 4(2,3), \\ 
\la_0 &:= \gamma_0(\hbfm_0) =  (1,2) + (2,1) + 2(1,0), \\
\la_{-1} &:= \gamma_{-1}(\hbfm_{-1}) = 3  (2,-1) + 2(1,-2) + (2,-3)
\end{align*}
By Proposition \ref{Prop: braid for CgO}, we have 
$$
\fR_{1} (  V(\la) ) \simeq \fR_{1} (  \Phi_\ddD(\hbfm) ) \simeq \Phi_\ddD( \bR_1(\hbfm) ) \simeq V(\la'),
$$
where $ \la' =  \gamma (\bR_1(\hbfm) )$. Example  \ref{Ex: braid} (ii) says that 
$
\la' = \la_2' + \la_1' + \la_0' + \la_{-1}', 
$ 
where 
\begin{align*}
	\la_2' &:= \gamma_1(\hbfm_2') = 5  (1,6), \\ 
	\la_1' &:= \gamma_1(\hbfm_1') = 2  (2,5) + 3 (1,4) + 2(2,3), \\ 
	\la_0' &:= \gamma_0(\hbfm_0') =  (1,2) + (2,1) + 2(1,0), \\
	\la_{-1}' &:= \gamma_{-1}(\hbfm_{-1}') = (2,-1) + 4(1,-2). 
\end{align*}

\ee

\end{example}

\end{document}